\documentclass[12pt]{amsart}
\usepackage{amssymb} 
\usepackage{latexsym} 
\usepackage[all]{xy}
\usepackage{url}
\usepackage{placeins}

\newcommand{\Z}{{\mathbb Z}}

\newcommand{\C}{{\mathbb C}}
\newcommand{\Q}{{\mathbb Q}}

\newcommand{\rhobar}{{\overline{\rho}}}

\newcommand{\PP}{{\mathbb P}}
\newcommand{\Aff}{{\mathbb A}}
\newcommand{\FF}{{\mathcal F}}

\newcommand{\QQ}{{\mathcal Q}}
\newcommand{\ta}{{\theta}}
\newcommand{\NN}{{\mathcal N}}

\newcommand{\Kbar}{\overline{K}}

\newcommand{\GL}{\operatorname{GL}}
\newcommand{\diag}{\operatorname{Diag}}
\newcommand{\coeff}{\operatorname{coeff}}

\newcommand{\SL}{\operatorname{SL}}
\newcommand{\PGL}{\operatorname{PGL}}
\newcommand{\Gal}{\operatorname{Gal}}

\newcommand{\Aut}{\operatorname{Aut}}

\newcommand{\Hom}{\operatorname{Hom}}
\newcommand{\trace}{\operatorname{trace}}

\newcommand{\Pic}{\operatorname{Pic}}

\newcommand{\eps}{\varepsilon}

\newcommand{\isom}{\cong}

\newcommand{\ra}{\longrightarrow}

\newcommand{\Div}{\operatorname{Div}}

\newcommand{\res}{\operatorname{res}}
\newcommand{\divv}{\operatorname{div}}

\newcommand{\F}{{\mathbb F}}
\newcommand{\G}{{\mathcal G}}

\newcommand{\vv}{{\mathbf v}}
\newcommand{\ww}{{\mathbf w}}
\newcommand{\yy}{{\mathbf y}}

\def\PSL{{\operatorname{PSL}}}

\def\la{{\lambda}}

\def\O{{\mathcal O}}

\newcommand{\cc}{{\mathfrak c}}
\newcommand{\vvv}{{\mathfrak v}}

                      {\hspace*{\fill}\nobreak$\Box$\par\medskip}
                       {\hspace*{\fill}\nobreak$\Box$\par\medskip}

\newtheorem{Proposition}{Proposition}[section]
\newtheorem{Theorem}[Proposition]{Theorem}
\newtheorem{Lemma}[Proposition]{Lemma}

\newtheorem{Question}[Proposition]{Question}

\theoremstyle{definition}

\newtheorem{Definition}[Proposition]{Definition}
\newtheorem{Remark}[Proposition]{Remark}

\newtheorem{Example}[Proposition]{Example}

\begin{document}
\date{23rd December 2019}
\title[On families of $13$-congruent elliptic curves]
{On families of $13$-congruent elliptic curves}

\author{T.A.~Fisher}
\address{University of Cambridge,
          DPMMS, Centre for Mathematical Sciences,
          Wilberforce Road, Cambridge CB3 0WB, UK}
\email{T.A.Fisher@dpmms.cam.ac.uk}

\renewcommand{\baselinestretch}{1.1}
\renewcommand{\arraystretch}{1.3}

\renewcommand{\theenumi}{\roman{enumi}}

\begin{abstract}
  We compute twists of the modular curve $X(13)$ that parametrise the
  elliptic curves $13$-congruent to a given elliptic curve.  Searching
  for rational points on these twists enables us to find non-trivial
  pairs of $13$-congruent elliptic curves over $\Q$, i.e. pairs of
  non-isogenous elliptic curves over $\Q$ whose $13$-torsion subgroups
  are isomorphic as Galois modules. We also find equations for the
  surfaces parametrising pairs of $13$-congruent elliptic
  curves. There are two such surfaces, corresponding to
  $13$-congruences that do, or do not, respect the Weil pairing. We
  write each as a double cover of the projective plane ramified over a
  highly singular model for Baran's modular curve of level $13$.  By
  finding suitable rational curves on these surfaces, we show that
  there are infinitely many non-trivial pairs of $13$-congruent
  elliptic curves over $\Q$.
\end{abstract}

\maketitle

\section{Introduction}

Elliptic curves $E$ and $E'$ are {\em $n$-congruent} if their
$n$-torsion subgroups are isomorphic as Galois modules. We say the
$n$-congruence has {\em power $k$} if the isomorphism raises the Weil
pairing to the power $k$.  Since multiplication-by-$m$, where $m$ is
an integer coprime to $n$, is an automorphism of the $n$-torsion
subgroup, we are only interested in $k \in (\Z/n\Z)^\times$ up to
multiplication by squares.  Taking $n = p$ an odd prime, we say the
congruence is {\em direct} if $k$ is a quadratic residue, and {\em
  skew} if $k$ is a quadratic non-residue.

The elliptic curves $n$-congruent with power $k$ to a given elliptic
curve $E$ are parametrised by (the non-cuspidal points of) the curve
$X_E(n,k)$. The pairs of elliptic curves that are $n$-congruent with
power $k$, up to simultaneous quadratic twist, are parametrised by (a
Zariski open subset of) the surface $Z(n,k)$.

If elliptic curves $E$ and $E'$ are related by an isogeny of degree
$d$, with $d$ coprime to $n$, then by standard properties of the Weil
pairing, $E$ and $E'$ are $n$-congruent with power $d$.  Congruences
of this form are said to be trivial.  We are interested in the
following two basic questions.
\begin{enumerate}
\item For which prime numbers $p$ do there exist non-trivial pairs
of $p$-congruent elliptic curves over $\Q$?
\item For which prime numbers $p$ do there exist infinitely many 
non-trivial pairs of $p$-congruent elliptic curves over $\Q$?
\end{enumerate}
To be more precise, in (ii) we ask for infinitely many pairs of
$j$-invariants, otherwise from any non-trivial pair of $p$-congruent
elliptic curves we could construct infinitely many by taking
simultaneous quadratic twists.

For $p=3,5$ we have $X_E(p,k) \isom \PP^1$ and so there are infinitely
many elliptic curves $p$-congruent to a given elliptic curve. Explicit
formulae for these families of elliptic curves are given in
\cite{RubinSilverberg} in the direct case, and in \cite{g1hess,enqI}
in the skew case.  For $p \ge 7$ the curves $X_E(p,k)$ have genus at
least $3$, and so by Faltings' theorem there are only finitely many
elliptic curves $p$-congruent to a given elliptic curve.  Kraus and
Oesterl\'e \cite{KO} gave the example of the directly $7$-congruent
elliptic curves $152a1$ and $7448e1$. This was extended to infinitely
many examples by Halberstadt and Kraus \cite{HK}, who also gave an
equation for $X_E(7,1)$. A modification of their method, due to
Poonen, Schaefer and Stoll \cite{PSS}, gives an equation for
$X_E(7,3)$, and from this we were able to exhibit in \cite{7and11}
infinitely many non-trivial pairs of skew $7$-congruent elliptic
curves.

Kani and Schanz \cite{KS} described the geometry of the surfaces
$Z(n,k)$, in particular showing that $Z(11,1)$ is an elliptic surface
of Kodaira dimension $1$. This work was extended by Kani and Rizzo
\cite{KR}, who showed there are infinitely many non-trivial pairs of
directly $11$-congruent elliptic curves.  We gave an alternative more
explicit proof of this fact in \cite{7and11}, and determined a
Weierstrass equation for the elliptic surface $Z(11,1)$ in
\cite{moduli}. Kumar~\cite[Theorem 21]{Ku} computed an equation for
$Z(11,2)$, and although not noted in his paper, the rational curve on
this surface given by $r s - r + s^2 - s + 1 = 0$ gives rise to
infinitely many non-trivial pairs of skew $11$-congruent elliptic
curves.

Examples of non-trivial $13$-congruent elliptic curves have been known
for some time. For example, the pair $52a1$ and $988b1$ appears in
\cite[Table 5.3]{FM}. Since the first of these curves admits a
rational $2$-isogeny, this gives an example of both a direct and a
skew 13-congruence. Prior to our work the only known examples of
non-trivial $13$-congruences were for pairs of elliptic curves that
are both in the range of Cremona's tables (or simultaneous quadratic
twists of such examples). In this paper, we show that there are
infinitely many non-trivial pairs of $13$-congruent elliptic curves,
both in the direct and skew cases.

The only known example of a $p$-congruence for $p > 13$ is the pair of
skew $17$-congruent elliptic curves $3675b1$ and $47775b1$. This
example was originally found by Cremona, and is explicitly recorded in
\cite{Billerey,CF,7and11,FK}. The fact the congruence is skew follows
from \cite[Proposition 2]{KO}. It is a conjecture of Frey and Mazur
that there are no non-trivial pairs of $p$-congruent elliptic curves
for $p$ sufficiently large. On the basis of our work, and that in
\cite{CF}, we might refine this conjecture by suggesting that the
answer to question (i) is the set of primes $p \le 17$, and the answer
to question (ii) is the set of primes $p \le 13$.
 
Another reason why the case $n=13$ is interesting, is that according
to Kani and Schanz \cite[Theorem 4]{KS} it is the smallest value of
$n$ for which all the surfaces $Z(n,k)$ are of general type.

In Section~\ref{sec:statres} we state our main results by giving
equations for $X_E(13,k)$ and $Z(13,k)$ for $k = 1,2$.  We also
describe some of the curves of small genus we found on the surfaces
$Z(13,k)$, including the ones giving rise to our infinite families of
non-trivial pairs of $13$-congruent elliptic curves.

To compute equations for $X_E(13,k)$ we follow the invariant-theoretic
method developed in \cite{7and11}. However we do more to explain the
generality in which we can expect these methods to work.  To compute
the necessary twists we need to start with an embedding of $X(p)$ in
projective space such that the group $\PSL_2(\Z/p\Z)$ acts linearly.
In Section~\ref{sec:X(p)} we explain the reasons behind our choice of
embedding (Klein's $A$-curve) in the case $p = 13$.  In
Section~\ref{sec:X(13)twists} we start on the invariant theory proper,
deriving equations first for $X(13)$, and then for its twists
$X_E(13,1)$ and $X_E(13,2)$. One basic difficulty is that the
invariant of smallest degree has degree~2, but since a quadratic form
has infinite automorphism group, it cannot carry the information
needed to specify our curve. This forced us to work with an invariant
of degree 4. The twisted forms of this invariant are too large to
sensibly include in the paper, but are available from
\cite{magmafile}.

Having equations for $X_E(13,k)$ in principle gives us equations for
$Z(13,k)$. However the equations obtained in this way are very
complicated, and not useful for finding rational points or curves of
small genus on the surfaces. For several smaller values of $n$, as
described in \cite{moduli}, we were able to find substitutions to
simplify these equations. However this step defeated us in the case
$n = 13$. In Section~\ref{sec:mdqs} we instead develop a new approach
for computing equations for $Z(n,k)$, not going via the equations for
$X_E(n,k)$, but still using the invariant theory.

Since the surface $Z(n,k)$ parametrises pairs of elliptic curves, it
comes with a standard involution that corresponds to swapping over the
two elliptic curves. The method in Section~\ref{sec:mdqs} gives us
equations for $Z(13,k)$ as a double cover of the plane, where the map
to the plane quotients out by the standard involution.  This is the
same format as used by Kumar \cite{Ku} when giving his equations for
$Z(n,-1)$ for $n \le 11$.  In using this format we are relying on the
fact that the quotient of $Z(n,k)$ by the standard involution is a
rational surface. It would be interesting to determine how large $n$
must become before this property fails.

In Section~\ref{sec:ex} we give some examples of pairs of non-trivial
$13$-congruent elliptic curves over $\Q$ and over $\Q(t)$. The
examples over $\Q$ may be verified, independently of our work, by
checking that the traces of Frobenius are congruent mod $13$ for
sufficiently many good primes $p$. The examples over $\Q(t)$ give
rise, by specialising~$t$, to the infinitely many examples over $\Q$
that are our main result. 

All computer calculations in support of this work was carried out
using Magma \cite{Magma}. Some Magma files containing some details of
the calculations are available from~\cite{magmafile}.  We refer to
elliptic curves by their labels in Cremona's tables~\cite{Cr}.  We
write $K$ for a field of characteristic $0$ and $\Kbar$ for its
algebraic closure.  

\section{Statement of results}
\label{sec:statres}

\subsection{The curves $X_E(13,1)$ and $X_E(13,2)$}
\label{sec:state-curves}
The elliptic curves $n$-congruent with power $k$ to a given elliptic
curve $E$ are parametrised by (the non-cuspidal points of) the smooth
projective curve $X_E(n,k)$.  We have computed equations for these
curves in the case $n = 13$.  In this section we give formulae first
for $X(13)$, and then for $X_E(13,1)$ and $X_E(13,2)$, each as a curve
of degree $42$ in $\PP^6$. Since the equations themselves would (in
the latter two cases) take several pages to write out, we instead
describe how they may be recovered from a set of $14$ hyperplanes,
equivalently a set of $14$ points in the dual projective space
$(\PP^6)^\vee$.  This description (which only uses linear algebra)
does not however correspond to how we originally computed the
equations.

In the case of $X(13)$ the $14$ points are 
\begin{equation}
\label{14pts-X13}
(1:0: \ldots : 0) \quad \text{ and } \quad (1:\zeta^k: \zeta^{3k} :
\zeta^{4k} :\zeta^{9k} :\zeta^{10k} :\zeta^{12k})
\end{equation}
where $\zeta = e^{2 \pi i/13}$ and $0 \le k \le 12$.

\begin{Theorem}
\label{thm:X13viapts}
Let $U$ be the $14$-dimensional space of quadratic forms vanishing at
the $14$ points \eqref{14pts-X13}. Let $U^\perp$ be the
$14$-dimensional space of quadratic forms annihilated by
\[ \left\{ f \left( \frac{\partial}{\partial x_1}, \ldots,
    \frac{\partial}{\partial x_7} \right) : f \in U \right\}. \]
Let $V \subset U^\perp$ be the $13$-dimensional subspace spanned by
the support of all linear syzygies, i.e. the span of the set
\[ \left\{ \sum_{i=1}^7 \lambda_i f_i \,\, \bigg| \,\, \lambda_i \in
  \{ 0,1 \}, f_i \in U^\perp \text{ and } \sum_{i=1}^7 x_i f_i = 0
\right\}. \]
Let $W$ be the $7$-dimensional space of cubic forms whose partial
derivatives belong to $V$. Then 
$W$ defines the union of $X(13) \subset \PP^6$ and $42$ lines.
\end{Theorem}

Our equations for $X_E(13,1)$ and $X_E(13,2)$ are obtained from those
for $X(13)$ by twisting, that is, by making a change of coordinates on
$\PP^6$ defined over $\Kbar$.  To describe the points that take the
place of~\eqref{14pts-X13}, we let $t$ be a coordinate on
$X_0(13) \isom \PP^1$ chosen (following Fricke) so that the $j$-map is
given by
\[ j = (t^2 + 5 t + 13) (t^4 + 7 t^3 + 20 t^2 + 19 t + 1)^3/t. \]
It is easy to write down an elliptic curve with this $j$-invariant.
For example, we may take the elliptic curve
$y^2 = x^3 - 27 c_4(t)x - 54 c_6(t)$ where
\begin{align*}
c_4(t) &= (t^2 + 5 t + 13) (t^2 + 6 t + 13)
          (t^4 + 7 t^3 + 20 t^2 + 19 t + 1), \\
c_6(t) &= (t^2 + 5 t + 13) (t^2 + 6 t + 13)^2
          (t^6 + 10 t^5 + 46 t^4 + 108 t^3 + 122 t^2 + 38 t - 1).
\end{align*}

We define polynomials $f_1, \ldots,f_7$ and $g_1,\ldots, g_7$ by
\begin{equation*}
\begin{aligned}
f_1(s,t) &= 1, \\
f_2(s,t) &= -(t + 1), \\
f_3(s,t) &= 3 s (t + 2) (t^2 + 5 t + 13) (t^2 + 6 t + 13), \\
f_4(s,t) &= -9 s (t + 1) (t^2 + 5 t + 13) (t^2 + 6 t + 13), \\
f_5(s,t) &= 108 s^2 (t^2 + 5 t + 13) (t^2 + 6 t + 13)
                         (t^3 + 5 t^2 + 10 t + 2) + 27 s^2 (t + 3) c_4(t), \\
f_6(s,t) &= 162 s^2 (t^2 + 5 t + 13) (t^2 + 6 t + 13)
                         (t^3 + 6 t^2 + 14 t + 7) + 27 s^2 (t + 4) c_4(t), \\
f_7(s,t) &= 11664 s^3 (t + 1) (t^2 + 5 t + 13) (t^2 + 6 t + 13)^2
                       + 54 s^2 c_4(t) f_4(s,t),
\end{aligned}
\end{equation*}
and
\begin{equation*}
\begin{aligned}
g_1(s,t) &= 2, \\
g_2(s,t) &= 2 (t + 1), \\
g_3(s,t) &= 3 s (t^2 + 6 t + 13) (t^3 + 4 t^2 + 8 t - 1), \\
g_4(s,t) &= 12 s (t^2 + 6 t + 13) (t^2 + 3 t + 5), \\
g_5(s,t) &= 6 s (t^2 + 6 t + 13) (t^3 + 8 t^2 + 20 t + 7), \\
g_6(s,t) &= 108 s^2 (t + 1)^2 (t^2 + 5 t + 13) (t^2 + 6 t + 13)
	           - 9 s^2 (t + 3) c_4(t), \\
g_7(s,t) &= -216 s^2 (t - 1) (t^2 + 5 t + 13) (t^2 + 6 t + 13)
                   - 18 s^2 (t + 2) c_4(t).  
\end{aligned}
\end{equation*}

\begin{Theorem}
\label{XE13thm}
Let $E/K$ be the elliptic curve $y^2 = x^3 + a x + b$.  Let $U_1$,
respectively $U_2$, be the space of quadratic forms vanishing at
\[ (f_1(s,t): \ldots : f_7(s,t)), \,\, \text{ respectively } \,\,
(g_1(s,t): \ldots : g_7(s,t)), \]
for all $s,t \in \Kbar$ satisfying $a = -27 s^2 c_4(t)$ and
$b = -54 s^3 c_6(t)$.  Let $W_k$ be the space of cubic forms
constructed from $U_k$ by the procedure in
Theorem~\ref{thm:X13viapts}. Then $W_k$ defines the union of
$X_E(13,k) \subset \PP^6$ and $42$ lines, where the latter are not in
general defined over $K$.
\end{Theorem}

The cubic forms in Theorem~\ref{XE13thm}, as polynomials with
coefficients in $\Z[a,b]$, are available from \cite{magmafile}.  As
described in Sections~\ref{sec:XE(13,1)} and~\ref{sec:XE(13,2)}, we
have also found equations that define the curve $X_E(13,k)$ exactly,
and define the $j$-map $X_E(13,k) \to \PP^1$.  It would be possible to
simplify the $f_i(s,t)$ and $g_i(s,t)$ by making a change of
coordinates on $\PP^6$.  However, we made our choice of co-ordinates
with the aim of simplifying the cubic forms.

\begin{Remark}
\label{rem:gradings}
If $a$ and $b$ have weights $2$ and $3$, and $x_1, \ldots, x_7$ have
weights $3,3,2,2,1,1,0$, then the cubic forms in the case $k=1$ are
homogeneous with weights $6, 7, 7, 8, 8, 9, 9$.  Likewise, if
$x_1, \ldots, x_7$ have weights $2,2,1,1,1,0,0$, then the cubic forms
in the case $k=2$ have weights $4, 5, 5, 6, 6, 6, 7$.  These gradings
reflect the fact that $X_E(13,k)$ only depends on $E$ up to quadratic
twist.
\end{Remark}

\subsection{The surfaces $Z(13,1)$ and $Z(13,2)$}
\label{state:Z}

The surface $Z(n,k)$ parametrises pairs of elliptic curves $E$ and
$E'$ that are $n$-congruent with power $k$, up to simultaneous
quadratic twist.  We have computed equations for these surfaces in the
case $n = 13$.
\begin{Theorem}
\label{thm:Zeqns}
(i) The surface $Z(13,1)$ is birational over $\Q$ to the surface with
affine equation $y^2 + h_1(r,s)y = g_1(r,s)$ where
\begin{align*}
h_1(r,s) &= s^4 + (2 r^2 - 5 r + 7) s^3 + (r^4 - 3 r^3 - 14 r^2 + r + 16) s^2 
  \\ & \qquad  + r^2 (2 r^3 - 5 r^2 + 15 r + 27) s + r^4 (r^2 - 1), \\
g_1(r,s) &= 4 (7 r - 8) s^6 + 22 (r - 2) s^5 - (28 r^5 + 24 r^4 - 2 r^3 -
    39 r^2 + 2 r + 68) s^4 \\ & \qquad + r^2 (84 r^3 + 233 r^2 - 116 r - 223) s^3
      - r^4 (20 r^2 + 181 r + 181) s^2 \\ & \qquad - 4 r^6 (r - 1) (7 r + 3) s.
\end{align*}
(ii) The surface $Z(13,2)$ is birational over $\Q$ to the surface with
affine equation $y^2 + h_2(r,s)y = g_2(r,s)$ where
\begin{align*}
h_2(r,s) &= r^3 s^4 + r (2 r^3 + 7 r^2 + 1) s^3 + (r^5 + 7 r^4 + 9 r^3 + r^2 +
    1) s^2 \\ & \qquad + 2 (r^3 + 2 r + 1) s + r + 1, \\
g_2(r,s) &= 2 r^4 (5 r + 4) s^7 + r^3 (19 r^3 + 48 r^2 + 33 r + 22) s^6
  \\ & \qquad + r^2 (8 r^5 + 40 r^4 + 79 r^3 + 82 r^2 + 47 r + 21) s^5
  \\ & \qquad - r (r^7 - 29 r^5 - 91 r^4 - 75 r^3 - 53 r^2 - 34 r - 7) s^4
  \\ & \qquad + r (6 r^6 + 35 r^5 + 50 r^4 + 37 r^3 + 42 r^2 + 22 r + 10) s^3
  \\ & \qquad + r (14 r^4 + 33 r^3 + 30 r^2 + 14 r + 1) s^2 + r^2 (10 r + 13) s + 2 r.
\end{align*}
(iii) Let $k = 1$ or $2$. Let $j,j' : Z(13,k) \to \PP^1$ be the maps
giving the $j$-invariants of the elliptic curves $E$ and $E'$.  We
have computed polynomials $A_k, B_k, D_k \in \Z[r,s]$ such that
$j j' = A^3_k/D_k$ and $(j - 1728)(j'-1728) = B^2_k / D_k$.  The
polynomials $A_k$ and $B_k$ are available from \cite{magmafile}.  The
$D_k$ are given by
\begin{align*}
  D_1(r,s) &= s^5 (r + s - 1)^4 (r^2 + s - 1)^2 (r^4 + r^3 s - r^3 + r s^2 - r s - s^2 + s)^{13},  \\
  D_2(r,s) &=      -r^6 (r^2 + r s + r + 1)^3 (r^3 s + r^2 s^2 + 2 r^2 s + r s^2 + r s + r + s)^{13}.
\end{align*}
\end{Theorem}

\begin{Remark} 
\label{completethesquare}
Let $k=1$ or $2$. By completing the square, the first two parts of
Theorem~\ref{thm:Zeqns} are equivalent to the statement that $Z(13,k)$
is birational to the surface $y^2 = F_k(r,s,1)$ where $F_k$ is the
homogeneous polynomial of degree $10+2k$ satisfying
$F_k(r,s,1) = h_k(r,s)^2 + 4 g_k(r,s)$.
\end{Remark} 

According to Kani and Schanz~\cite{KS}, the surfaces $Z(13,k)$ are of
general type, and so by the Bombieri-Lang conjecture (see for example
\cite{HS}) are expected to contain only finitely many curves of genus
$0$ or $1$.

On $Z(13,1)$ there are genus $0$ curves given by the vanishing of $s$,
$r + s - 1$, $r^2 + s - 1$, $r$ and $r + s$. The first three of these
are factors of $D_1$, and so do not correspond to any families of
elliptic curves. The last two define copies of the modular curves
$X_0(10)$ and $X_0(25)$. Remarkably we found a further pair of
genus~$0$ curves given by
\[r^5 + r^4 s - 3 r^4 - r^3 s + 2 r^2 s^2 - 4 r^2 s - 2 r s^2 + s^3 -
4 s^2 = 0.\]
From this we obtain the infinite family of directly 13-congruent
elliptic curves presented in Example~\ref{QTex:dir}.  There are also
genus $1$ curves given by the vanishing of $r^2 + s$,
$r^2 + r s - r - s + 1$ and $r^2 + r s - s$. These are copies of
$X_0(m)$ for $m= 27,36$ and $49$.

On $Z(13,2)$ there are genus $0$ curves given by the vanishing of $r$,
$r^2 + r s + r + 1$, $s$ and $r^2 s + r s^2 + r s + 2 s^2 - 2 s + 1$.
The first two are factors of $D_2$, the third is a copy of $X_0(18)$,
and from the fourth we obtain the infinite family of skew 13-congruent
elliptic curves presented in Example~\ref{QTex:skew}. There are also
genus $1$ curves given by the vanishing of $r + 1$, $s - 1$,
$r s + 1$, $r^2 s + 2 r s + 1$ and $r^2 s + r s + 1$.  These are
copies of $X_0(m)$ for $m = 19,20,21,24$ and $32$. A further genus $1$
curve is given by
\[ r^3 s + r^2 s^2 + 3 r^2 s + 4 r s + r + 2 = 0. \]
This is an elliptic curve of rank $2$ with Cremona label $267632f1$
and Weierstrass equation $y^2 = x^3 - 515 x - 4494$. It parametrises
another infinite family of non-trivial pairs of skew $13$-congruent
elliptic curves.

It would be interesting to determine whether there are any more curves
of genus 0 or 1 on the surfaces $y^2 = F_k(r,s,1)$.

\subsection{Baran's modular curve}
\label{sec:baran}

For $k=1,2$ we have written $Z(13,k)$ as a double cover of $\PP^2$
ramified over the curve $C_{k} = \{ F_{k} = 0 \}$.  A rational point
on $C_{k}$ corresponds to an elliptic curve that is $13$-congruent to
itself in a non-trivial way. Such a congruence is only possible if the
mod 13 Galois representation of the elliptic curve is not
surjective. More specifically, arguing as in \cite{Halberstadt}, we
see that $C_1$ and $C_2$ are copies of the modular curves
$X_{\rm s}^+(13)$ and $X_{\rm ns}^+(13)$ associated to the normaliser
of a split or non-split Cartan subgroup of level $13$.

These curves were first computed by Baran \cite{Baran}, who also
discovered the surprising fact, specific to level $13$, that the two
curves are isomorphic. We were able to verify using Magma that our
singular curves $C_1$ and $C_2$ (of degrees $12$ and $14$) are both
birational to the smooth plane quartic
\[ C = \{ (y + z) x^3 - (2 y^2 + y z) x^2 + (y^3 - y^2 z + 2 y z^2 -
z^3) x - (2 y^2 z^2 - 3 y z^3) = 0 \} \subset \PP^2 \]
found by Baran. Using Theorem~\ref{thm:Zeqns}(iii) we were also able
to recover the two different moduli interpretations of this curve, as
given in \cite[Appendix A]{Baran}.  We remark that the determination
of all $\Q$-rational points on $C$ (and hence also on $C_1$ and $C_2$)
was recently completed in \cite{cursedcurve}.

We describe further modular curves on the surfaces $Z(13,k)$ in
Section~\ref{sec:modcrvs}.

\section{Twists and quotients}
\label{sec:twq}

In this section we recall the definition of $X(n)$ over a
non-algebraically closed field, and explain how in principle
$X_E(n,k)$ may be described as a twist of $X(n)$.  We also describe
$Z(n,k)$ as a quotient of $X(n) \times X(n)$.  We write $\zeta_n$ for
a primitive $n$th root of unity, and $\mu_n$ for the group of all
$n$th roots unity.

Let $n \ge 3$ be an integer. The modular curve $X(n)$ is the smooth
projective curve birational to $Y(n)$, where $Y(n)$ is the modular
curve parametrising the pairs $(E,\phi)$ where $E$ is an elliptic
curve and $\phi : E[n] \to \mu_n \times \Z/n\Z$ is a symplectic
isomorphism. By symplectic we mean that the Weil pairing on $E[n]$
agrees with the standard pairing
$((\zeta,c),(\xi,d)) \mapsto \zeta^d \xi^{-c}$ on
$\mu_n \times \Z/n\Z$. We note that, with this definition, $X(n)$ is
both defined over $\Q$ and geometrically irreducible.

Let $\Gamma$ be the group of symplectic automorphisms of 
$\mu_n \times \Z/n\Z$. As a group this is a copy of $\SL_2(\Z/n\Z)$,
but with Galois action given by
\begin{equation}
\label{twistedgal}
 \sigma(\gamma) = 
\begin{pmatrix} \chi(\sigma) & 0 \\ 0 & 1 \end{pmatrix} \gamma
\begin{pmatrix} \chi(\sigma)^{-1} & 0 \\ 0 & 1 \end{pmatrix} 
\end{equation}
where $\chi$ is the mod $n$ cyclotomic character.  There is an action
of $\Gamma$ on $X(n)$ given by
$\gamma : (E,\phi) \mapsto (E, \gamma \phi)$.  We suppose that
\begin{enumerate}
\item we have embedded $X(n) \subset \PP^{N-1}$, and
\item the action of $\Gamma$ is given by a Galois equivariant
group homomorphism 
\[ \rho : \Gamma \to \GL_N( \Q(\zeta_n) ). \]
\end{enumerate}

The following is a variant of~\cite[Lemma 3.2]{7and11}.  We write
$\sigma_k$ for the automorphism of $\Q(\zeta_n)$ given by
$\zeta_n \mapsto \zeta_n^k$. We also write $\propto$ for equality in
$\PGL_N$.
 
\begin{Lemma}
\label{lem:twist}
Let $E/K$ be an elliptic curve and
$\phi : E[n] \to \mu_n \times \Z/n\Z$ a symplectic isomorphism defined
over $\Kbar$. Suppose $h \in \GL_N(\Kbar)$ satisfies
\begin{equation}
\label{cocycle1}
  \sigma(h) h^{-1} \propto \sigma_k \rho ( \sigma(\phi) \phi^{-1} ) 
\end{equation}
for all $\sigma \in \Gal(\Kbar/K)$. Then $X_E(n,k) \subset \PP^{N-1}$
is the twist of $X(n) \subset \PP^{N-1}$ given by
$X_E(n,k) \isom X(n); \, {\bf x} \mapsto h {\bf x}$.
\end{Lemma}

\begin{proof}
  Let $\eps_k : \mu_n \times \Z/n\Z \to \mu_n \times \Z/n\Z$ be the
  map sending $(\zeta,b) \mapsto (\zeta^k,b)$.  The non-cuspidal
  points of $X_E(n,k)$ correspond to pairs $(F,\psi)$ where $F$ is an
  elliptic curve and $\psi : F[n] \to E[n]$ is an isomorphism that
  raises the Weil pairing to the power $k^{-1}$.  (In fact we could
  take the power to be $k m^2$ for any $m \in (\Z/n\Z)^\times$, but
  the choice here is convenient for the definition of $\alpha$.)  Let
  $\alpha : X_E(n,k) \to X(n)$ be given by
  $(F,\psi) \mapsto (F,\eps_k \phi \psi)$. Then
\begin{equation}
\label{cocycle2}
 \sigma(\alpha) \alpha^{-1} \propto \rho( \sigma(\eps_k \phi)
(\eps_k \phi)^{-1}) = \rho (\eps_k \sigma(\phi) \phi^{-1} \eps_k^{-1})
 = \sigma_k \rho( \sigma(\phi) \phi^{-1}), 
\end{equation}
where for the last two equalities we have used~\eqref{twistedgal} and
the fact that both $\eps_k$ and $\rho$ are Galois equivariant.

Now let $X' = \{ {\bf x} \in \PP^{N-1} : h {\bf x} \in X(n) \}$. Since
$\sigma(h) h^{-1}$ is an automorphism of $X(n)$ we see that $X'$ is
defined over $K$. By~\eqref{cocycle1} and~\eqref{cocycle2} the curves
$X_E(n,k)$ and $X'$ are twists of $X(n)$ by the same cocycle, and are
therefore isomorphic over $K$.
\end{proof}

The following description of $Z(n,k)$ as a quotient of
$X(n) \times X(n)$ is the starting point of \cite{KS}.  We revisit
this result since we wish to be sure that it works over a
non-algebraically closed field.

\begin{Lemma} 
\label{lem:quot}
The surface $Z(n,k)$ is birational to the quotient of
$X(n) \times X(n)$ by the action of $\Gamma$ given by
$\gamma \mapsto ( \rho(\gamma), \sigma_k \rho(\gamma) )$.
\end{Lemma}

\begin{proof}  
  There is a Galois equivariant map $X(n) \times X(n) \to Z(n,k)$
  given by
  \[ ((E_1,\phi_1), (E_2,\phi_2)) \mapsto (E_1,E_2,\phi_2^{-1} \eps_k
  \phi_1). \]
  where $\eps_k$ is as in the proof of Lemma~\ref{lem:twist}.  If we
  act by $\gamma \in \Gamma$ then $\phi_1$ and $\phi_2$ become
  $\gamma \phi_1$ and $\eps_k \gamma \eps_k^{-1} \phi_2$. This leaves
  $\phi_2^{-1} \eps_k \phi_1$ unchanged. Conversely, any pair of
  points in $Y(n) \times Y(n)$ with the same image in $Z(n,k)$ are
  related in this way.
\end{proof}

\section{The modular curve $X(p)$}
\label{sec:X(p)}

In this section we explain how (in the case $n=p$ is a prime) we may
arrange that the assumptions (i) and (ii) in Section~\ref{sec:twq} are
satisfied. We also describe the ring of invariants that arises in this
context.

\subsection{Group actions on curves}
Let $X$ be a smooth projective curve over $\C$, and let $G$ be a
finite group of automorphisms of $X$. Let $G$ act trivially on
$\C^\times$ and on $\C(X)^\times$ by
$\sigma : f \mapsto f \circ \sigma^{-1}$.  Splitting the exact
sequence of $G$-modules
\[ 0 \to \C^\times \to \C(X)^\times \to \Div X \to \Pic X \to 0 \]
into short exact sequences, and taking group cohomology gives
a diagram
\begin{equation*}
\begin{aligned}
\xymatrix{ && H^1(G,\C(X)^\times) \ar[d] \\ 
(\Div X)^G
\ar[r]  & (\Pic X)^G \ar[r]^-{\delta} & H^1(G,\C(X)^\times/\C^\times) \ar[d]^-{\Delta} \\
&& H^2(G,\C^\times) } 
\end{aligned}
\end{equation*}
Let $\Upsilon: (\Pic X)^G \to H^2(G,\C^\times)$ be the composite of
the connecting maps $\delta$ and $\Delta$.  Since $G$ acts faithfully
on $X$, we have $H^1(G,\C(X)^\times) = 0$ by Hilbert's theorem~90.  We
thus obtain an exact sequence
\begin{equation}
\label{basic-exseq}
 0 \ra \frac{(\Div X)^G}{\sim} \ra (\Pic X)^G 
           \stackrel{\Upsilon}{\ra} H^2(G, \C^\times). 
\end{equation}

There is an alternative description of $\Upsilon$ in terms of theta
groups.  For $D \in \Div X$ representing an element of $(\Pic X)^G$ we
define the theta group
\[ \Theta_D = \{ (f,\sigma) : f \in \C(X)^\times, \sigma \in G 
\text{ such that } \divv(f) = \sigma D - D \} \]
with group law
\begin{equation}
\label{th:eqn1}
 (f,\sigma) \circ (g,\tau) = (f \cdot \sigma(g), \sigma \tau). 
\end{equation}
This group sits naturally in an exact sequence
$0 \to \C^\times \to \Theta_D \to G \to 0$.  In other words,
$\Theta_D$ is an extension of $G$ by $\C^\times$.

\begin{Lemma}
\label{lem:Ob}
If $[D] \in (\Pic X)^G$ then $\Upsilon(D)$ is the class of $\Theta_D$
in $H^2(G,\C^\times)$.
\end{Lemma}
\begin{proof} For each $\sigma \in G$ we pick
  $f_\sigma \in \C(X)^\times$ with $\divv(f_\sigma) = \sigma D - D$.
  The class of $\Theta_D$ in $H^2(G,\C^\times)$ is represented by the
  $2$-cocycle $\phi$ satisfying
\begin{equation}
\label{th:eqn2}
 (f_\sigma,\sigma) \circ (f_\tau,\tau) = \phi(\sigma,\tau)  (f_{\sigma \tau}, \sigma \tau). 
\end{equation}
Comparing~\eqref{th:eqn1} and~\eqref{th:eqn2} we find that
$\phi(\sigma,\tau) = f_\sigma \cdot \sigma(f_{\tau}) \cdot f_{\sigma
  \tau} ^{-1}$.
By the formulae for the connecting maps in group cohomology, we see
that the image of $[D]$ under $\delta$ is represented by
$\sigma \mapsto f_\sigma$, and its image under $\Delta$ is represented
by $\phi$.
\end{proof}

\begin{Lemma}
\label{lem:lift}
If $[D] \in (\Pic X)^G$ and $H^0(X,\O(D))$ has dimension $n \ge 1$,
then there is a natural action of $G$ on the $1$-dimensional subspaces
of $H^0(X,\O(D))$ giving rise to a projective representation
$\rhobar : G \to \PGL_n(\C)$. This lifts to a representation
$\rho : G \to \GL_n(\C)$ if and only if $\Upsilon(D) = 0$.
\end{Lemma}
\begin{proof}
  There is a linear action of $\Theta_D$ on $H^0(X,\O(D))$ via
  $(f,\sigma) : g \mapsto f \cdot \sigma(g)$. Picking a basis for
  $H^0(X,\O(D))$, this defines a representation
  $\pi : \Theta_D \to \GL_n(\C)$. There is a commutative diagram with
  exact rows
  \[ \xymatrix{ 0 \ar[r] & \C^\times \ar@{=}[d] \ar[r] & \Theta_D
    \ar[d]_\pi \ar[r] & G
    \ar@{.>}[dl] \ar[d]^{\rhobar} \ar[r] & 0 \\
    0 \ar[r] & \C^\times \ar[r] & \GL_n(\C) \ar[r] & \PGL_n(\C) \ar[r]
    & 0 } \]
  By Lemma~\ref{lem:Ob} we have $\Upsilon(D) = 0$ if and only if the
  top row splits. If the top row splits then it is clear that
  $\rhobar$ lifts to $\rho$ (as indicated by the dotted arrow).
  Conversely if $\rhobar$ lifts to $\rho$, then by a diagram chase
  each $\sigma \in G$ lifts uniquely to $x \in \Theta_D$ with
  $\pi(x) = \rho(\sigma)$, and the map $\sigma \mapsto x$ is a
  splitting of the top row.
\end{proof}

\subsection{The action of $\PSL_2(\Z/p\Z)$ on $X(p)$}
We shall need the following standard group-theoretic facts.
\begin{Lemma} 
\label{gp:lem}
Let $\G = \SL_2(\Z/p\Z)$ where $p \ge 5$ is a prime. Then
\begin{enumerate}
\item The groups $H^i(\G,\C^\times)$ are trivial for $i=1,2$.
\item Every projective representation of $\G$
lifts uniquely to a representation.
\end{enumerate}
\end{Lemma}
\begin{proof}
  (i) The group $\G$ is generated by elements
  $S = (\begin{smallmatrix} 0 & 1 \\ -1 & 0 \end{smallmatrix})$ and
  $T = (\begin{smallmatrix} 1 & 1 \\ 0 & 1 \end{smallmatrix})$ with
  $S^4 = T^p = (ST)^3 = I_2$. Therefore
  $H^1(\G,\C^\times) = \Hom(\G,\C^\times) = 0$. The vanishing of
  $H^2(\G,\C^\times)$ was proved by Schur, using the fact that every
  Sylow subgroup of $\G$ is either cyclic or a generalised quaternion
  group.  See \cite[Theorem 4.232]{Gorenstein} or~\cite[Chapter V,
  Satz 25.7]{Huppert}. \\
  (ii) If $\rhobar : \G \to \PGL_n(\C)$ is a projective representation
  then
  \[ \{ (g,M) \in \G \times \GL_n(\C) : \rhobar(g) \propto M \} \]
  is an extension of $\G$ by $\C^\times$, and so corresponds to an
  element of $H^2(\G,\C^\times)$.  Thus the vanishing of
  $H^2(\G,\C^\times)$ proves the existence of a lift, and the
  vanishing of $\Hom(\G,\C^\times)$ shows it is unique.
\end{proof}

Now let $X = X(p)$ where $p \ge 5$ is a prime. As a Riemann surface,
it is the quotient of the extended upper half plane
${\mathfrak H}^* = {\mathfrak H} \cup \PP^1(\Q)$ by the action of
$\Gamma(p) = \ker(\SL_2(\Z) \to \SL_2(\Z/p\Z))$. There is an action of
$G = \PSL_2(\Z/p\Z)$ on $X$ with quotient the $j$-line. The quotient
map is ramified over $j = 0,1728,\infty$ with ramification indices 3,
2 and $p$. Thus, writing $\nu = (p^2 -1)/24$, all but three $G$-orbits
of points on $X$ have size $|G| = 12 p \nu$, and the remaining orbits
have sizes $12 \nu$, $4p \nu$ and $6p \nu$. It may be proved using the
Hurwitz bound (see \cite[Theorem~20.40]{AR}) that $G$ is the full
automorphism group of $X$ when $p \ge 7$.

The character table of $\G = \SL_2(\Z/p\Z)$ is described for example
in~\cite[\S5.2]{FH}.  The non-trivial representations of smallest
degree are conjugate representations $\phi$ and $\phi'$ each of degree
$m = (p-1)/2$, and conjugate representations $\psi$ and $\psi'$ each
of degree $m + 1$.  Klein gave equations for $X(p)$ both as a curve of
degree $(m-1)\nu$ in $\PP^{m-1}$ with $\G$ acting via $\phi$, and as a
curve of degree $m \nu$ in $\PP^{m}$ with $\G$ acting via $\psi$.
Following the terminology in \cite{AR}, we call these the $z$-curve
and the $A$-curve. For example, when $p=7$ the $z$-curve is the Klein
quartic.

\begin{Theorem}[Adler, Ramanan]
\label{thm:AR}
The group $(\Pic X)^G$ is infinite cyclic, generated by a divisor
class $\lambda$ of degree $\nu = (p^2 -1)/24$.
\end{Theorem}
\begin{proof}  
  This is \cite[Theorem 24.1]{AR}.  The proof works by analysing the
  exact sequence~\eqref{basic-exseq}.  The authors first show that
  $(\Div X)^G/\sim$ is infinite cyclic, generated by a divisor class
  of degree $\gcd(12 \nu, 4 p \nu, 6 p \nu) = 2 \nu$.  By
  Lemma~\ref{gp:lem}(i) and the Hochschild-Serre exact sequence
  \[ \Hom(\G,\C^\times) \stackrel{\res}{\ra} \Hom(\{\pm 1\},\C^\times)
  \ra H^2(G,\C^\times) \stackrel{\inf}{\ra} H^2(\G,\C^\times). \]
  we have $H^2(G, \C^\times) \isom \Z/2\Z$.  The proof is completed by
  constructing $\lambda$ as the difference of the hyperplane sections
  for the $z$-curve and the $A$-curve.
\end{proof}
Applying the Riemann Hurwitz theorem to the $j$-map $X(p) \to \PP^1$
shows that $X(p)$ has genus $(p-6)\nu + 1$. The canonical divisor is
therefore $2(p-6) \lambda$.

\subsection{An abstract ring of invariants}

We introduce a ring that plays a central role in our calculations.

\begin{Theorem}
\label{thm:RG}
Let $R = \oplus_{d\ge 0} R_d = \oplus_{d\ge 0} H^0(X, \O(d \lambda))$ 
and $\G = \SL_2(\Z/p\Z)$.
\begin{enumerate} 
\item There is a natural action of $\G$
on $R$ where $-I_2$ acts as $(-1)^d$ on $R_d$.
\item The $\G$-invariant subring of $R$ is generated 
by elements $c_4$, $c_6$ and $D$ of degrees $4p$, $6p$ and $12$.
\item We may scale $c_4$, $c_6$ and $D$ so that 
$c_4^3 - c_6^2 = 1728 D^p$ and the $j$-map
$X \to \PP^1$ is given by $j = c_4^3/ D^p$.
\end{enumerate}
\end{Theorem}

\begin{proof}
  (i) Suppose that $R_d = H^0(X, \O(d \lambda))$ has dimension
  $n \ge 1$.  Since $\lambda$ is $G$-invariant, we obtain a projective
  representation $\overline{\rho} : G \to \PGL_n(\C)$, and by
  Lemma~\ref{gp:lem}(ii) this lifts uniquely to a representation
  $\rho : \G \to \GL_n(\C)$.  This gives the required action of $\G$
  on $R_d$. It is clear that $\rho(-I_2) = \pm I_n$.
  Lemma~\ref{lem:lift} shows that the sign is $+$ (i.e., the action
  factors via $G$) precisely when $\Upsilon(d \lambda)=0$. However we
  saw in the proof of Theorem~\ref{thm:AR} that $\Upsilon(\lambda)$ is
  the non-trivial element of $H^2(G, \C^\times) \isom \Z/2\Z$. The
  action of $\G$ on $R_d$
  therefore factors via $G$ precisely when $d$ is even. \\
  (ii) The fibres of the $j$-map above $0, 1728$ and $\infty$ are
  effective divisors in the classes of $4p \lambda$, $6p \lambda$ and
  $12 \lambda$.  We let $c_4$, $c_6$ and $D$ be the corresponding
  elements of $R$.  Let $f \in R_d$ be a $\G$-invariant element.  We
  show by induction on $d$ that $f$ belongs to the subring generated
  by $c_4$, $c_6$ and $D$. If $d \ge 1$ then $f$ vanishes on the
  $G$-orbit of some point $P \in X$.  If the orbit has size $4p\nu$,
  $6p\nu$ or $12 \nu$ then we divide through by $c_4$, $c_6$ or $D$,
  and apply the induction hypothesis.  Otherwise the orbit has size
  $|G|= 12p \nu$. In this case we divide through by a linear
  combination of $c_4^3$ and $c_6^2$
  chosen so that it vanishes at $P$. \\
  (iii) Let $P \in X$ be a cusp, i.e., a point above $j = \infty$.
  Let $f$ be a linear combination of $c_4^3$ and $c_6^2$ that vanishes
  at $P$. Since $f$ vanishes at exactly $|G|$ points (counted with
  multiplicity) it cannot vanish on any orbits of size $|G|$.
  Therefore $f$ vanishes only at the cusps, and so must be a scalar
  multiple of $D^p$. Scaling the invariants appropriately gives the
  relation as claimed.  Finally, the formula offered for the $j$-map
  quotients out by the action of $G$, and has degree $|G|$. It must
  therefore agree with the $j$-map up to composition with a Mobius
  map. However, since both maps send the zeros of $c_4, c_6$ and $D$
  to $j = 0, 1728$ and $\infty$, this Mobius map fixes three points,
  and is therefore the identity.
\end{proof}

In our earlier work \cite{7and11} on twists of $X(p)$ for $p=7$ and
$11$, we mainly worked with the $z$-curve. In the case $p=13$ the
$z$-curve has degree $35$ in $\PP^5$ and the $A$-curve has degree $42$
in $\PP^6$. By Theorem~\ref{thm:RG} we have
\begin{align}
\nonumber
\oplus_{d \ge 0} H^0(X, \O(5d \lambda))^{\G} &= \C[ D^5, D c_6, D^4 c_4, c_4 c_6, D^3 c_4^2 ],\\ 
\label{az13}
\oplus_{d \ge 0} H^0(X, \O(6d \lambda))^{G} &= \C[ D, c_6].
\end{align}
The ring of invariants is much simpler in the second of these two
cases. We therefore decided to work with the $A$-curve in the case
$p = 13$.

\section{Equations for $X(13)$ and its twists}
\label{sec:X(13)twists}

\subsection{Equations for the $A$-curve}
\label{sec:X13}
Let $\zeta = e^{2\pi i/13}$ and $\xi_k = \zeta^k + \zeta^{-k}$.  Let
$G \isom \PSL_2(\Z/13\Z)$ be the subgroup of $\SL_7(\C)$ generated by
$M_2$, $M_6$ and $M_{13}$ where

\[ M_2 = \frac{1}{\sqrt{13}} \begin{pmatrix} 
    1 & 1 & 1 & 1 & 1 & 1 & 1 \\
    2 & \xi_{2} & \xi_{4} & \xi_{8} & \xi_{3} & \xi_{6} & \xi_{1} \\
    2 & \xi_{4} & \xi_{8} & \xi_{3} & \xi_{6} & \xi_{1} & \xi_{2} \\
    2 & \xi_{8} & \xi_{3} & \xi_{6} & \xi_{1} & \xi_{2} & \xi_{4} \\
    2 & \xi_{3} & \xi_{6} & \xi_{1} & \xi_{2} & \xi_{4} & \xi_{8} \\
    2 & \xi_{6} & \xi_{1} & \xi_{2} & \xi_{4} & \xi_{8} & \xi_{3} \\
    2 & \xi_{1} & \xi_{2} & \xi_{4} & \xi_{8} & \xi_{3} & \xi_{6} 
 \end{pmatrix}, \quad  
M_6 = -\begin{pmatrix}
1 & 0 & 0 & 0 & 0 & 0 & 0 \\
 0 & 0 & 0 & 0 & 0 & 0 &1 \\
 0 &1 & 0 & 0 & 0 & 0 & 0 \\
 0 & 0 &1 & 0 & 0 & 0 & 0 \\
 0 & 0 & 0 &1 & 0 & 0 & 0 \\
 0 & 0 & 0 & 0 &1 & 0 & 0 \\
 0 & 0 & 0 & 0 & 0 &1 & 0 
\end{pmatrix},
\]
and $M_{13}= \diag(1,\zeta,\zeta^4,\zeta^3,\zeta^{12},\zeta^{9},\zeta^{10})$.
We write $\C[x_0,\ldots,x_6]_d$ for the space of homogeneous
polynomials of degree $d$.

\begin{Definition}
\label{def:inv}
An {\em invariant} of degree $d$ is a polynomial
$I \in \C[x_0,\ldots,x_6]_d$ satisfying $I \circ g = I$ for all
$g \in G$.
\end{Definition}

The invariants of smallest degree are $Q$ and $F$ given by
\begin{align*}
 Q &= x_0^2 + x_1 x_4 + x_2 x_5 + x_3 x_6, \\
 F &= 2x_0^4 + 6 x_0(x_1 x_3 x_5 + x_2 x_4 x_6)
 + 3 (x_1x_2x_4x_5 + x_1x_3x_4x_6 + x_2x_3x_5x_6)
 \\ & \qquad 
    + x_1 x_2^3 + x_2 x_3^3 + x_3 x_4^3 + x_4 x_5^3 + x_5 x_6^3 + x_6 x_1^3. 
\end{align*}
We use these invariants to give equations for $X(13)$ as a curve of
degree 42 in $\PP^6$, defined over $\Q$.  For $f$ and $g$ homogeneous
forms in $x_0,\ldots,x_6$ we put
\begin{equation}
\label{brac}
\langle f,g \rangle = \trace ( H(f) H(Q)^{-1} H(g) H(Q)^{-1} )
\end{equation}
where $H$ denotes the Hessian matrix, that is, the $7 \times 7$ matrix of
second partial derivatives. We prove the following refinement
of Theorem~\ref{thm:X13viapts}.

\begin{Theorem}
\label{thm:QF->curve}
Let $Q$ and $F$ be the invariants defined above. 
\begin{enumerate}
\item The vector space $W$ of cubic forms $f$ satisfying
  $\langle f, F - 3Q^2 \rangle = 0$ has dimension $7$.  Moreover
  $F - 3Q^2$ is, up to scalars, the unique quartic form satisfying
  $\langle f, F - 3Q^2 \rangle = 0$ for all $f \in W$.
\item Let $U$ be the vector space of quadratic forms vanishing on the
  $G$-orbit of $\{x_0 = 0\}$ in $(\PP^6)^\vee$.  Then $W$ is the space
  of cubic forms constructed from $U$ by the procedure in
  Theorem~\ref{thm:X13viapts}.
\item If $W$ has basis $f_0, \ldots, f_6$ then
  \[ X(13) \isom \{ f_0 = \ldots = f_6 = F + Q^2 = 0 \} \subset
  \PP^6. \]
  This is a curve of degree $42$, and the $84$ cusps are cut out by
  the quadratic form $Q$. The cubic forms $f_0, \ldots, f_6$ are not
  sufficient to define the curve, but rather define the union of the
  curve and $42$ lines. The $42$ lines each pass through two cusps,
  and may be divided into $14$ sets of $3$, where each set of $3$
  lines spans one of the hyperplanes in (ii).
\end{enumerate}
\end{Theorem}
\begin{proof}
The first two parts are checked by linear algebra. The space
of cubic forms $W$ has basis $f_0, \ldots,f_6$ where
\begin{align*}
f_0 &= -2 x_0^3 + x_0 (x_1 x_4 + x_2 x_5 + x_3 x_6) + x_1 x_3 x_5 + x_2 x_4 x_6, \\
f_1 &= x_0 x_1^2 + 2 x_0 x_3 x_4 + 2 x_1 x_2 x_6  + x_2 x_4^2 + x_5 x_3^2 + x_6 x_5^2,
\end{align*}
and the remaining $f_i$ are obtained from $f_1$ by the action of 
$M_6$, i.e., by cyclically permuting 
the subscripts $1, 2, \ldots, 6$. 

Let $a_1, \ldots, a_6$ be coordinates on $\PP^5$. We write $a_0 = 0$,
$a_{-i} = -a_i$ and agree to read all subscripts mod 13. According to
\cite[Section 2]{7and11}, the $z$-curve for $X(13)$ is defined by the
4 by 4 Pfaffians of the 13 by 13 skew symmetric matrix
$(a_{i-j} a_{i+j})$.  According to \cite[\S 51]{AR}, the $A$-curve is
the image of the $z$-curve via the map
\[ (x_0:x_1: \ldots: x_6) = \left( 1: \frac{a_2}{a_1}:
  \frac{a_4}{a_2}: \frac{a_8}{a_4}: \frac{a_3}{a_8}:\frac{a_6}{a_3}:
  \frac{a_{12}}{a_6} \right). \]
A calculation, performed using Magma~\cite{Magma}, shows that the
$A$-curve is defined by the vanishing of $f_0, \ldots,f_6$ and
$F + Q^2$.  As we remark in the proof of the next lemma, further
equations are needed to generate the homogeneous ideal. We also
checked using Magma that this curve has degree 42, and that it meets
the hypersurface defined by $Q$ in 84 distinct points. This set of
points is preserved by the action of $G \isom \PSL_2(\Z/13\Z)$, and so
must be the set of cusps on $X(13)$.

If we write $P_0=(1:0 \ldots : 0)$, $P_1 = (0:1:0: \ldots : 0)$, etc,
then $P_1, P_2, \ldots, P_6$ are cusps, and the cubics vanish on the
lines $P_1P_4$, $P_2P_5$ and $P_3P_6$. These lines belong to a single
$G$-orbit of size $42$. Another calculation using Magma shows that the
cubics define a curve of degree $84$, which must therefore be the
union of $X(13)$ and the $42$ lines.
\end{proof}




Some care must be taken in working with the above model for $X(13)$,
since it is not projectively normal. In other words, the rings $S$ and
$S'$ in the following lemma are not the same.
\begin{Lemma}
\label{lem:notpn}
Let $X = X(13) \subset \PP^6$ be as in Theorem~\ref{thm:QF->curve}.
Let $S = \oplus_{d \ge 0} S_d$ be its homogeneous coordinate ring, and
let $S' = \oplus_{d \ge 0} H^0(X,\O_X(d))$. Then
\begin{align*}
\sum_{d \ge 0} (\dim S_d) t^d = 1 + 7 t + 28 t^2 + 77 t^3 + 119 t^4 + \ldots \\
\sum_{d \ge 0} (\dim S'_d) t^d = 1 + 7 t + 35 t^2 + 77 t^3 + 119 t^4 + \ldots 
\end{align*}
and $\dim S_d = \dim S'_d= 42 d - 49$ for all $d \ge 3$.
\end{Lemma}

\begin{proof}
  Using the Gr\"obner basis machinery in Magma \cite{Magma} we were
  able to compute $42$ quartic forms that together with the $7$ cubic
  forms generate the homogeneous ideal of $X$. From this it is easy to
  compute $\dim S_d$ for any given $d$.  In particular we verified the
  values recorded in the statement of the lemma for each $d \le 4$. On
  the other hand, since $X$ has degree $42$ and genus $50$ it follows
  by Riemann-Roch that $\dim S'_d = 42d - 49$ for all $d \ge 3$.

  Let $T = \oplus_d T_d$ be the homogeneous coordinate ring of the set
  of $84$ cusps. Again by computer algebra we find
  \[ \sum_{d \ge 0} (\dim T_d) t^d = 1 + 7 t + 27 t^2 + 70 t^3 + 84
  t^4 + \ldots \]
  Therefore $\dim T_d = 84$ for all $d \ge 4$.  We show by induction
  on $d$ that the inclusion $S_d \subset S'_d$ is an equality for all
  $d \ge 3$. We have already checked this for $d= 3,4$.  So let
  $f \in S'_d$ with $d \ge 5$. Since $\dim T_d = 84$ we may reduce to
  the case where $f$ vanishes at the cusps. But then applying the
  induction hypothesis to $f/Q \in S'_{d-2}$ gives the result.
  Finally, by identifying $S'_d$ with the subspace of $S'_{d+2}$
  vanishing at the cusps, we compute
  \begin{align*}
    \dim S'_1 &= \dim S_3 - \dim T_3 = 7, \\
    \dim S'_2 &= \dim S_4 - \dim T_4 = 35. \qedhere 
  \end{align*}
\end{proof}

\begin{Remark}
  We have shown that $\dim S'_1 = 7$ and therefore
  $X(13) \subset \PP^6$ is embedded by a complete linear system. This
  is a special case of the ``WYSIWYG hypothesis'' in \cite{AR}.
\end{Remark}

\begin{Definition}
\label{def:cov}
A {\em covariant} of degree $d$ is a column vector $\vv$ of
polynomials in $\C[x_0,\ldots,x_6]_d$ satisfying $\vv \circ g = g \vv$
for all $g \in G$.
\end{Definition}

Starting from an invariant $I$ of degree $d$ we may construct a
covariant of degree $d-1$ as
\[ \nabla_Q I = H(Q)^{-1} \begin{pmatrix} \partial I/ \partial x_0 \\
\vdots \\ \partial I/ \partial x_6 \end{pmatrix}. \]
On the other hand if $\vv$ and $\ww$ are covariants of degrees $d$ and $e$ then 
\[ \vv \cdot \ww := \vv^T H(Q) \ww = \coeff( Q(\vv + t \ww) , t) \]
is an invariant of degree $d+e$. We may also think of covariants as
$G$-equivariant polynomial maps $\C^7 \to \C^7$. Thus the composition
of covariants $\vv$ and $\ww$ of degrees $d$ and $e$ is a covariant
$\vv \circ \ww$ of degree $de$.

It is easy to compute the dimensions of the spaces of invariants and
covariants of any given degree $d$ from the character table of $G$.
We may also solve for the invariants and covariants of degree $d$ by
linear algebra over $\Q(\zeta)$, at least if $d$ is not too large.

\begin{Remark}
\label{rem:effavg}
A standard trick for computing invariants is to start with an
arbitrary polynomial $f$ and apply the operator
$f \mapsto \frac{1}{|G|} \sum_{g \in G} f \circ g$.  An efficient way
to organise this calculation, which will become important in
Section~\ref{sec:mdqs}, is the following.  Let
$\pi : f \mapsto \frac{1}{13} \sum_{i=0}^{12} f \circ M_{13}^i$ be the
projection map that sends a monomial fixed by $M_{13}$ to itself, and
all other monomials to zero. Then starting with a monomial in the
image of $\pi$ we apply the operators
$f \mapsto \sum_{i=0}^5 f \circ M_6^i$ and
$f \mapsto f + 13 \pi(f \circ M_2)$. 
\end{Remark}

For our work in Sections~\ref{sec:XE(13,1)} and~\ref{sec:XE(13,2)}, it
is important that we find ways of constructing invariants and
covariants from previously known examples in a basis-free way. Another
consideration is that it would be overkill for us to completely
classify the invariants and covariants, as we are only interested in
them modulo the equations defining the curve $X = X(13)$.

Writing $\psi$ for the standard representation of
$G \subset \SL_7(\C)$, we find that $\wedge^3 \psi$ contains a copy of
the trivial representation.  The corresponding $G$-invariant
alternating form is
\begin{equation}
\label{altform}
\Phi = (x_0 \wedge x_1 \wedge x_4) - (x_0  \wedge x_2 \wedge x_5) 
+ (x_0 \wedge x_3 \wedge x_6) 
+ (x_1 \wedge x_3 \wedge x_5) - (x_2 \wedge x_4 \wedge x_6).  
\end{equation}

Again let $\langle ~,~ \rangle$ be as defined in~\eqref{brac}.

\begin{Lemma} 
\label{lem:N}
There is a $7 \times 7$ alternating matrix $N$ of linear forms in
$\C[x_0,\ldots,x_6]$, unique up to an overall scaling, such that
\[ \big\langle ( N \nabla_Q F )_i, \frac{\partial F}{\partial x_i} \big\rangle = 0 \]
for all $0  \le i \le 6$. 
\end{Lemma} 
\begin{proof}
This is checked by linear algebra. We find that
\begin{equation}
\label{eqn:N}
N = \begin{pmatrix}
     0 & x_4 & -x_5 & x_6 & -x_1 & x_2 & -x_3 \\
     -x_4 & 0 & 0 & x_5 & x_0 & -x_3 & 0 \\
     x_5 & 0 & 0 & 0 & -x_6 & -x_0 & x_4 \\
     -x_6 & -x_5 & 0 & 0 & 0 & x_1 & x_0 \\
     x_1 & -x_0 & x_6 & 0 & 0 & 0 & -x_2 \\
     -x_2 & x_3 & x_0 & -x_1 & 0 & 0 & 0 \\
     x_3 & 0 & -x_4 & -x_0 & x_2 & 0 & 0 
\end{pmatrix}. 
\end{equation}
In a more succinct notation, we have $N = (N_{ij})$ where 
$N_{ij} = (\frac{\partial}{\partial x_i} 
       \wedge \frac{\partial}{\partial x_j}) \Phi$.
\end{proof}

We define covariants $\vv_3 = \nabla_Q F$ and
$\vv_4 = H(Q)^{-1} N \vv_3$, where $N$ is given by~\eqref{eqn:N}.
Then $\vv_9 = \vv_3 \circ \vv_3$ is a covariant of degree 9, and
$c_6 = \vv_4 \cdot \vv_9$ is an invariant of degree 13.  Our next
theorem shows that although the rings $S$ and $S'$ in
Lemma~\ref{lem:notpn} are different, their $G$-invariant subrings are
the same.

\begin{Theorem}
\label{myinv}
Let $S$ be the coordinate ring of $X \subset \PP^6$. Then
$S^G = \C[Q,c_6]$ and the $j$-map $X \to \PP^1$ is given by
$j = 1728 - c_6^2/Q^{13}$.
\end{Theorem}
\begin{proof} We find that $c_6(0,1,0,0,0,0,0) = -1$, and so $c_6$
  does not vanish identically on $X$. Since $X \subset \PP^6$ is the
  $A$-curve, it has hyperplane section $6 \lambda$. Therefore $S^G$ is
  a subring of
  \[ \oplus_{d \ge 0} H^0(X, \O(6d \lambda))^G. \]
  By Theorem~\ref{thm:RG}, or more specifically~\eqref{az13}, the
  latter is a polynomial ring in two variables, generated in degrees 2
  and 13.  Since we have constructed invariants $Q$ and $c_6$ of these
  degrees, and these invariants do not vanish on $X$, this proves the
  first part of the theorem.

  By Theorem~\ref{thm:RG}(iii) we have $j - 1728 = \xi c_6^2/Q^{13}$
  for some constant $\xi$. Let $\omega$ be a primitive cube root of
  unity, and put $\alpha = \sqrt{-1 + 3 \omega}$.  The point
  \[
  (-2:\omega+\alpha:\omega-\alpha:\omega+\alpha:\omega-\alpha:\omega+\alpha:\omega-\alpha)
  \in X \]
  is fixed by $M_6^2 \in G$ of order $3$, and so lies above $j = 0$.
  The function $c_6^2/Q^{13}$ takes the value $1728$ at this point.
  Therefore $\xi = -1$.
\end{proof}

For later use (when applying Lemma~\ref{lem:23}) we also record a
point on $X$ above $j=1728$.  Let $i = \sqrt{-1}$ and let $\beta$ be a
root of $x^3 - (i + 1) x^2 - x + i = 0$. Let $\sigma$ be the
automorphism of $\Q(\beta)$ given by $\beta \mapsto \beta^2 - \beta$.
Then the point
\[ (1 : \beta : \sigma(\beta) : \sigma^2(\beta) : \beta :
\sigma(\beta) : \sigma^2(\beta) ) \in X \]
is fixed by $M_6^3 \in G$ of order $2$, and so lies above $j = 1728$.

\subsection{Equations for $X_E(13,1)$}
\label{sec:XE(13,1)}

We compute equations for $X_E(13,1)$ from those for $X = X(13)$ in
Theorem~\ref{thm:QF->curve} by making a change of coordinates. For
this we use the $7 \times 7$ matrix formed from the following $7$
covariants. As before, the subscripts indicate the degrees of the
covariants.
\begin{align*}
\vv_1 &= (x_0,x_1, \ldots,x_6)^T &
\vv_{10} &= \coeff(\vv_4 \circ (\vv_1 + t \vv_3), t^3) \\
\vv_3 &= \nabla_Q F &
\vv_{12} &= \vv_3 \circ \vv_4 \\
\vv_4 &= H(Q)^{-1} N \vv_3 & 
\vv_{13} &= \coeff(\vv_4 \circ (\vv_1 + t \vv_4), t^3) 
\\
\vv_9 &= \vv_3 \circ \vv_3 
\end{align*}

\begin{Lemma} 
\label{lem:det1}
We have 
\[ \det(\vv_1,\vv_3,\vv_4,\vv_9,\vv_{10},\vv_{12},\vv_{13}) 
 = (c_6^2 - 1728 Q^{13})^2 \mod{ I(X) }. \]
\end{Lemma}
\begin{proof} The left hand side is an invariant of degree 52, and so
  by Theorem~\ref{myinv} is a linear combination of $Q^{26}$,
  $Q^{13}c_6^2$ and $c_6^4$.  We may determine the correct linear
  combination by evaluating each side at some random points on $X$. We
  initially did this by working mod $p$ for some moderately sized
  prime $p$. To verify the answer in characteristic $0$, we used the
  point on $X$ defined over a degree 20 number field given by
  $(1 : 1 : \gamma : \ldots )$ 
  where $\gamma$ is a root of
\begin{align*}
 x^{20} +  5 x^{17} - 7 & x^{16} + 2 x^{15} + 10 x^{14} + x^{13} + 5 x^{12} +
    4 x^{11} - 21 x^{10} \\ & + 19 x^9 + 10 x^8 + 3 x^7 + 8 x^6 - 17 x^5 +
    5 x^3 + 2 x^2 + 1 = 0.  \qedhere
\end{align*} 
\end{proof}

We now twist the invariants $Q$ and $F$. Specifically we put
\begin{align*}
{\QQ}(y_1, \ldots, y_7) &= Q(y_1 \vv_1 + y_2  \vv_3 + y_3 \vv_4 + y_4 \vv_9
  + y_5 \vv_{10} + y_6 \vv_{12} + y_7 \vv_{13}) \\
{\mathcal F}(y_1, \ldots, y_7) &= F(y_1 \vv_1 + y_2  \vv_3 + y_3 \vv_4 + y_4 \vv_9
  + y_5 \vv_{10} + y_6 \vv_{12} + y_7 \vv_{13})
\end{align*}
Each of the coefficients of these forms is an invariant, and so
working modulo $I(X)$ may be written as a polynomial in $Q$ and $c_6$.
By a series of computations similar to the proof of
Lemma~\ref{lem:det1} we find that
\begin{align*}
\QQ(y_1, &\ldots, y_7) = Q y_1^2 - 4 Q^2 y_1 y_2 + 44 Q^5 y_1 y_4 
   + c_6 y_1 y_6 - 36 Q^7 y_1 y_7 - 2 Q^3 y_2^2 \\ &
   - 124 Q^6 y_2 y_4 - c_6 y_2 y_5 + 300 Q^8 y_2 y_7 + 
    6 Q^4 y_3^2 + c_6 y_3 y_4 + 324 Q^7 y_3 y_5 \\ & - 12 Q^8 y_3 y_6
   - 2 Q^2 c_6 y_3 y_7 + 502 Q^9 y_4^2 + 24 Q^3 c_6 y_4 y_5 + 
    27 Q^4 c_6 y_4 y_6 \\ & - 396 Q^{11} y_4 y_7 
   + 4302 Q^{10} y_5^2 - 36 Q^{11} y_5 y_6 
    - 35 Q^5 c_6 y_5 y_7 + 150 Q^{12} y_6^2 \\ & 
    - 54 Q^6 c_6 y_6 y_7 - (3282 Q^{13} - c_6^2) y_7^2.
\end{align*}
The coefficient of $y_1^2$ is $\vv_1 \cdot \vv_1 = Q$, and the
coefficient of $y_3y_4$ is $\vv_4 \cdot \vv_9 = c_6$, which is how we
defined $c_6$ in the last section. We note that several of the
coefficients were forced to be zero by the fact there are no monomials
in $Q$ and $c_6$ of the appropriate degree.  In a similar way we
compute
\begin{align*}
\FF(y_1, &\ldots, y_7) =  -Q^2 y_1^4 - 4 Q^3 y_1^3 y_2 - 
    124 Q^6 y_1^3 y_4 - c_6 y_1^3 y_5 + 300 Q^8 y_1^3 y_7 \\ & - 
    6 Q^4 y_1^2 y_2^2 - 372 Q^7 y_1^2 y_2 y_4 - 
    3 Q c_6 y_1^2 y_2 y_5 + 900 Q^9 y_1^2 y_2 y_7 + 18 Q^5 y_1^2 y_3^2 \\
    & + \ldots + (307161 Q^{19} c_6 - 
    32 Q^6 c_6^3) y_6 y_7^3 - (24003375 Q^{26} + 408 Q^{13} c_6^2 - 
    2 c_6^4) y_7^4.
\end{align*}

We now make a change of coordinates to simplify $\QQ$ and $\FF$, and
so that we obtain correct formulae in the case $j(E) = 0$.  We put
$c_6 = -864b$ and substitute
\begin{align*}
y_1 &=   16 Q^6 (12 a^2 (x_1 + 2 x_2) + 18 a b x_3 + 14 (4 a^3 + 27 b^2) x_5
    + 81 b^2 x_6), \\
y_2 &=   2 Q^5 (96 a^2 x_2 - 144 a b (x_3 - 3 x_4) + (48 a^3 - 324 b^2) x_5
   \\ & \qquad\qquad\qquad\qquad\qquad\qquad
 + (52 a^3 - 297 b^2) x_6  - 864 a^2 b x_7), \\
y_3 &=   Q^{11} (2 a (x_3 - 56 x_4) + 3 b (44 x_5 + 3 x_6) + 176 a^2 x_7), \\
y_4 &=   -4 Q^2 (4 a^3 + 27 b^2) (2 x_5 - x_6), \\
y_5 &=   2 Q^8 (2 a x_4 - 3 b x_5 - 4 a^2 x_7), \\
y_6 &=   Q^7 (2 a (x_3 - 2 x_4) + 3 b (2 x_5 + 3 x_6) + 8 a^2 x_7), \\
y_7 &=   2 (4 a^3 + 27 b^2) x_6.
\end{align*}
This transformation has determinant
$-2^{23} 3^3 a^8 Q^{39} (4 a^3 + 27 b^2)^2$.  Dividing $\QQ$ and $\FF$
by $2^{16} 3^2 a^4 (4 a^3 + 27 b^2)$ and
$2^{32} 3^4 a^8 (4 a^3 + 27 b^2)^2$, and eliminating $Q$ by the rule
$Q^{13} = 16(4 a^3 + 27b^2)$, we obtain
\begin{align*}
\QQ(x_1, &\ldots, x_7) = 
   x_1^2 - 6 x_2^2 + a x_3^2 + 9 b x_3 x_6 - 6 a x_4^2 + 18 b x_4 x_5 +
  24 a^2 x_4 x_7 \\ & + 2 a^2 x_5^2 - 36 a b x_5 x_7 - 3 a^2 x_6^2 + 162 b^2 x_7^2,
\end{align*}
 and
\begin{align*}
\FF(x_1, & \ldots, x_7) =
   -x_1^4 - 12 x_1^3 x_2 - 54 x_1^2 x_2^2 - 18 a x_1^2 x_4^2 +
  54 b x_1^2 x_4 x_5 \\ & + 72 a^2 x_1^2 x_4 x_7 + 6 a^2 x_1^2 x_5^2 -
  108 a b x_1^2 x_5 x_7 + 486 b^2 x_1^2 x_7^2 - 60 x_1 x_2^3 \\ & + \ldots
   -432 a b (4 a^3 - 27 b^2) x_6 x_7^3 
   -3 (8 a^3 - 27 b^2)(40 a^3 - 27 b^2) x_7^4.
\end{align*}
These polynomials $\QQ$ and $\FF$ have weights $6$ and $12$ with
respect to the grading in Remark~\ref{rem:gradings}.

For $f$ and $g$ homogeneous polynomials in $x_1, \ldots,x_7$ we define
  \[ \langle f,g \rangle = \trace ( H(f) H(\QQ)^{-1} H(g) H(\QQ)^{-1} ).\]
The proof of the next theorem is similar to that of \cite[Lemmas 3.7
and 3.12]{7and11}.

\begin{Theorem}
\label{thm:XE13-dir}
Let $E/K$ be the elliptic curve $y^2 = x^3 +a x+ b$.  Let
$f_1,\ldots,f_7$ be a basis for the space of cubic forms $f$
satisfying $\langle f, \FF - 3 \QQ^2 \rangle = 0$.  Then
\[X_E(13,1) \isom \{ f_1 = \ldots = f_7 = \FF + \QQ^2 = 0 \} \subset
\PP^6.\]
\end{Theorem}
\begin{proof}
  We assume that $j(E) \not=0,1728$, equivalently $ab \not= 0$,
  leaving the remaining cases to Section~\ref{sec:j=0,1728}.  Let
  $(x_0: \ldots : x_6)$ be a $\Kbar$-point on $X(13)$ corresponding to
  $(E,\phi)$ for some choice of symplectic isomorphism
  $\phi : E[13] \to \mu_{13} \times \Z/13\Z$. By Theorem~\ref{myinv},
  and the formula $j(E) = 1728(4a^3)/(4a^3 + 27b^2)$, we may scale
  $(x_0, \ldots, x_6)$ to satisfy
  $Q(x_0, \ldots,x_6)^{13} = 16(4a^3 + 27b^2)$ and
  \begin{equation} \label{scalings} c_6(x_0,\ldots, x_6) =
    -864b. \end{equation}

  Let $h$ be the $7 \times 7$ matrix formed by evaluating the
  covariants
  \begin{equation}
    \label{covlist} 
    Q^6 \vv_1, \, Q^5 \vv_3, \, Q^{11} \vv_4, \, Q^2 \vv_9, \, Q^8 \vv_{10}, \, Q^7 \vv_{12}, \, \vv_{13} 
  \end{equation}
  at $(x_0, \ldots, x_6)$.  By Lemma~\ref{lem:det1} and our assumption
  $a \not= 0$ this matrix is non-singular. Let
  $\rho : \SL_2(\Z/13\Z) \to \GL_7(\Kbar)$ describe the action of
  $\PSL_2(\Z/13\Z)$ on $X(13)$. We claim that
  \begin{equation}
    \label{claim}
    \sigma(h) h^{-1} \propto \rho (\sigma(\phi) \phi^{-1}) 
  \end{equation}
  for all $\sigma \in \Gal(\Kbar/K)$. Let
  $\xi_\sigma =\sigma(\phi) \phi^{-1} \in \SL_2(\Z/13\Z)$.  Since
  $(\sigma(x_0): \ldots : \sigma(x_6))$ corresponds to
  $(E, \sigma(\phi))$ we have
  \begin{equation}
    \label{cocycle}
    \begin{pmatrix} \sigma(x_0) \\ \vdots \\ \sigma(x_6) \end{pmatrix} 
    = \lambda_\sigma \rho(\xi_\sigma) \begin{pmatrix} x_0 \\ \vdots \\ x_6 \end{pmatrix} 
  \end{equation}
  for some $\lambda_\sigma \in \Kbar^\times$. Then since $c_6$ is an
  invariant of degree $13$ we have
  \[ \sigma(c_6(x_0, \ldots,x_6)) = \lambda_\sigma^{13}
  c_6(x_0,\ldots,x_6). \]
  By~\eqref{scalings} and our assumption $b \not=0$, it follows that
  $\lambda_\sigma$ is a 13th root of unity. Since the
  covariants~\eqref{covlist} all have degree a multiple of $13$, our
  claim~\eqref{claim} now follows from~\eqref{cocycle} and the
  definition of a covariant (see Definition~\ref{def:cov}).

  Finally, Lemma~\ref{lem:twist} shows that $X_E(13,1) \subset \PP^6$
  is obtained from $X(13) \subset \PP^6$ by making the change of
  coordinates given by $h$, and Theorem~\ref{thm:QF->curve} shows how
  we may recover equations for the curve from $\QQ$ and $\FF$.
\end{proof}

The cubic forms $f_1, \ldots, f_7$, as polynomials with coefficients
in $\Z[a,b]$, are available from~\cite{magmafile}. Alternatively, they
may be computed from the description given in
Section~\ref{sec:state-curves}.

Next we compute the $j$-map $X_E(13,1) \to \PP^1$.  Revisiting
Lemma~\ref{lem:N} with $\QQ$ and $\FF$ in place of $Q$ and $F$ we
obtain a skew symmetric matrix $\NN$ given by
$\NN_{ij} = (\frac{\partial}{\partial x_i} \wedge
\frac{\partial}{\partial x_j}) \Phi$ where
\begin{align*} \Phi 
  &=  12 (x_1 \wedge x_2 \wedge x_7) - 2 (x_1 \wedge x_3 \wedge x_5) 
    - 3 (x_1 \wedge x_3 \wedge x_6) + 6 (x_1 \wedge x_4 \wedge x_5) \\ 
  & + 6 (x_1 \wedge x_4 \wedge x_6) + 12 a (x_1 \wedge x_5 \wedge x_7) 
    + 12 a (x_1 \wedge x_6 \wedge x_7) - 6 (x_2 \wedge x_3 \wedge x_5) \\ 
  & - 6 (x_2 \wedge x_3 \wedge x_6) + 12 (x_2 \wedge x_4 \wedge x_5) 
    + 18 (x_2 \wedge x_4 \wedge x_6) + 24 a (x_2 \wedge x_5 \wedge x_7) \\ 
  & + 36 a (x_2 \wedge x_6 \wedge x_7) - 12 a (x_3 \wedge x_4 \wedge x_7) 
    + 18 b (x_3 \wedge x_5 \wedge x_7) + 54 b (x_4 \wedge x_6 \wedge x_7) \\ 
  & + 12 a^2 (x_5 \wedge x_6 \wedge x_7).
\end{align*}
We put $\vvv_3 = \nabla_{\!\QQ}\, \FF$, $\vvv_4 = H(\QQ)^{-1} \NN \vvv_3$,
$\vvv_{9} = \vvv_3 \circ \vvv_3$ and
$\cc_6 = \coeff ( \QQ( \vvv_4 + t \vvv_9) , t)$.
The map $j : X_E(13,1) \to \PP^1$ satisfies $j - 1728 = \xi \cc_6^2/\QQ^{13}$
for some constant $\xi$. Evaluating at the 
tautological point $(1:0: \ldots :0) \in X_E(13,1)$ we find
\[ \QQ(1,0, \ldots, 0) = 1 \quad \text{ and } \quad \cc_6(1,0, \ldots,
0) = -216 b/(4 a^3 + 27 b^2). \]
Since $j(E) = 1728(4a^3)/(4a^3 + 27b^2)$ it follows that
$\xi = -(4 a^3 + 27 b^2)$ and so
\[ j = 1728 - \frac{(4 a^3 + 27 b^2) \cc_6^2}{\QQ^{13}}. \]

\begin{Remark}
\label{rem:g2}
(i) The quadratic form $\QQ$ may be recovered from $\Phi$ as
the GCD of the $6 \times 6$ Pfaffians of $\NN$ where
$\NN_{ij} = (\frac{\partial}{\partial x_i} 
       \wedge \frac{\partial}{\partial x_j}) \Phi$. \\
(ii) We may simplify $\QQ$ and $\Phi$ by making the further change 
of coordinates
\begin{align*}
x_1 &=  -a u_1 - a^2 u_2 + 9 b u_3 + 9 b u_4 
          - (3/2) b u_5 - 3 a u_6 + 2 a^2 u_7, \\
x_2 &=  (1/2) a u_1 + (1/3) a^2 u_2
          - (9/2) b u_3 - 3 b u_4 + (3/4) b u_5 + a u_6 - a^2 u_7, 
\end{align*}
$x_3 =  3 b u_2 + 2 a u_4$, 
$x_4 =  2 a u_3 - 3 b u_7$, 
$x_5 =  -(1/2) u_1 - a u_7$, 
$x_6 =  -(1/3) a u_2 + u_6$, 
$x_7 =  (1/2) u_3 + (1/12) u_5$.
Then $\QQ = u_1 u_7 + u_2 u_6 + u_3 u_5 + u_4^2$ and
\begin{equation}
\label{Phi:k=1}
 \Phi = (u_1 \wedge u_2 \wedge u_3) + (u_1 \wedge u_4 \wedge u_7) 
        + (u_2 \wedge u_4 \wedge u_6) + (u_3 \wedge u_4 \wedge u_5)
        + (u_5 \wedge u_6 \wedge u_7). 
\end{equation}
In particular these expressions do not depend on $a$ and $b$.
Unfortunately, this change of coordinates
makes $\FF$ more complicated. \\
(iii) The alternating forms~\eqref{altform} and~\eqref{Phi:k=1} differ
by a relabelling of the variables.  In computing our equations for
$X_E(13,1)$ from those for $X(13)$ we have therefore twisted by a
cocycle taking values in the stabiliser of $\Phi$. According to
\cite[Proposition~22.12]{FH} this stabiliser is the $14$-dimensional
exceptional Lie group $G_2$.
\end{Remark}

\subsection{Equations for $X_E(13,2)$}
\label{sec:XE(13,2)}

Let $G \isom \PSL_2(\Z/13\Z)$ be the subgroup of $\SL_7(\C)$ defined
in Section~\ref{sec:X13}.  We write $g \mapsto \widetilde{g}$ for the
automorphism of $G$ induced by $\zeta \mapsto \zeta^2$.

\begin{Definition}
\label{def:skewcov}
A {\em skew covariant} of degree $d$ is a column vector $\ww$ of
polynomials in $\C[x_0,\ldots,x_6]_d$ satisfying
$\ww \circ g = \widetilde{g} \ww$ for all $g \in G$.
\end{Definition}

Our first example of a skew covariant is
$\ww_3 = (f_0,f_1, \ldots, f_6)^T$ where the $f_i$ are the cubic
polynomials vanishing on $X(13)$, as defined in the proof of
Theorem~\ref{thm:QF->curve}.  Let $\ww_4 = (g_0,g_1, \ldots,g_6)^T$ be
the skew covariant of degree $4$ where
\begin{align*}
g_0 &= 4 x_0 (x_1 x_3 x_5 - x_2 x_4 x_6)
        + x_1 x_2^3 - x_2 x_3^3 + x_3 x_4^3 - x_4 x_5^3 + x_5 x_6^3 - x_6 x_1^3, \\
g_1 &= 4 x_0^2 x_1^2 - 4 x_0^2 x_3 x_4 + 4 x_0 x_1 x_2 x_6 - 2 x_0 x_3^2 x_5 - 2 x_0 x_5^2 x_6 - 
    x_1^3 x_4 + x_1^2 x_2 x_5 \\ & - 2 x_1 x_3 x_4^2 - x_1 x_5^3 - x_2^3 x_3 - 2 x_2^2 x_6^2 + 
    4 x_2 x_3 x_4 x_5 + x_3^2 x_4 x_6 + x_4 x_5 x_6^2,
\end{align*}
and the remaining $g_i$ are obtained from $g_1$ by the action of
$M_6$, i.e. by cyclically permuting the subscripts $1,2, \ldots, 6$
and alternating the signs.  We note that the polynomials
$g_0,g_1, \ldots,g_6$ vanish at the cusps of $X(13)$, but do not
vanish identically on $X(13)$, and are not divisible by $Q$.  They
therefore account for the discrepancy (in degree $2$) between the
rings $S$ and $S'$ in Lemma~\ref{lem:notpn}.

We construct further skew covariants by precomposing with a covariant.
\begin{align*}
\ww_5 &= \coeff(\ww_3 \circ (\vv_1 + t \vv_3) ,t) \qquad & 
\ww_8 &= \coeff(\ww_4 \circ (\vv_1 + t \vv_3) ,t^2)\\
\ww_6 &= \coeff(\ww_4 \circ (\vv_1 + t \vv_3) ,t) & 
\ww_{11} &= \coeff(\ww_5 \circ (\vv_1 + t \vv_3) ,t^3)\\
\ww_7 &= \coeff(\ww_3 \circ (\vv_1 + t \vv_3) ,t^2) & 
\ww_{13} &= \coeff(\ww_5 \circ (\vv_1 + t \vv_3) ,t^4)
\end{align*}

\begin{Lemma} 
\label{lem:det2}
We have 
\[ 
\det (\ww_4, \ww_5, \ww_6, \ww_7, \ww_8, \ww_{11}, \ww_{13} )
 = 2 Q (c_6^2 - 1728 Q^{13})^2 \mod{ I(X) }. \]
\end{Lemma}
\begin{proof} The proof is similar to that of Lemma~\ref{lem:det1}.
  The factor $Q$ on the right hand side arises since the entries of
  $\ww_4$ vanish at the cusps.
\end{proof}

We now twist the invariants $Q$ and $F$. Specifically we put
\begin{align*}
{\QQ}(y_1, \ldots, y_7) &= Q(y_1 \ww_4 + y_2  \ww_5 + y_3 \ww_6 + y_4 \ww_7
  + y_5 \ww_8 + y_6 \ww_{11} + y_7 \ww_{13}) \\
{\mathcal F}(y_1, \ldots, y_7) &= F(y_1 \ww_4 + y_2  \ww_5 + y_3 \ww_6 + y_4 \ww_7
  + y_5 \ww_8 + y_6 \ww_{11} + y_7 \ww_{13})
\end{align*}
Each of the coefficients of these forms is an invariant, and so
working modulo $I(X)$ may be written as a polynomial in $Q$ and $c_6$.
By a series of computations similar to the proof of
Lemma~\ref{lem:det1} we find that
\begin{align*}
\QQ(y_1, &\ldots, y_7) =
2 Q^4 y_1^2 + 16 Q^5 y_1 y_3 - 72 Q^6 y_1 y_5 + Q c_6 y_1 y_6 + 6 Q^5 y_2^2 \\
  & - 2 c_6 y_2 y_5 - 864 Q^8 y_2 y_6 + 1932 Q^9 y_2 y_7 - 4 Q^6 y_3^2 + c_6 y_3 y_4 - 
    72 Q^7 y_3 y_5 \\ & - 20 Q^2 c_6 y_3 y_6 - 42 Q^3 c_6 y_3 y_7 - 12 Q^7 y_4^2 - 
    3 Q c_6 y_4 y_5 + 576 Q^9 y_4 y_6 \\ & + 1008 Q^{10} y_4 y_7 + 612 Q^8 y_5^2 + 
    198 Q^3 c_6 y_5 y_6 - 196 Q^4 c_6 y_5 y_7 \\ & + 24408 Q^{11} y_6^2 - 
    163296 Q^{12} y_6 y_7 + (132630 Q^{13} + c_6^2) y_7^2,
\end{align*}
and 
\begin{align*}
\FF(y_1, &\ldots, y_7) = 
5 Q^8 y_1^4 + 8 Q^9 y_1^3 y_3 + Q^3 c_6 y_1^3 y_4 - 144 Q^{10} y_1^3 y_5 - 
    19 Q^5 c_6 y_1^3 y_6 \\ & - 42 Q^6 c_6 y_1^3 y_7 + 54 Q^9 y_1^2 y_2^2 - 
    6 Q^3 c_6 y_1^2 y_2 y_3 + 144 Q^{10} y_1^2 y_2 y_4 \\ 
& + \ldots + (8864253225 Q^{26} - 1969502 Q^{13} c_6^2 + 2 c_6^4) y_7^4.
\end{align*}

We put $c_6 = -864b$ and substitute
\begin{align*}
  y_1 &=  6 Q^{11} (36 b x_1 + 90 b x_2 - 2 a^2 (5 x_3 + 8 x_4 + 2 x_5) - 9 a b x_6 + 42 a b x_7), \\
  y_2 &=  4 Q^4 (4 a^3 + 27 b^2) (1347 x_1 - 936 x_2 - 317 a x_6 - 185 a x_7), \\
  y_3 &=  3 Q^{10} (36 b x_1 - 18 b x_2 + 2 a^2 (x_3 - 8 x_4 - 14 x_5) + 9 a b x_6 + 6 a b x_7), \\
  y_4 &=  24 Q^3 (4 a^3 + 27 b^2) (21 x_1 - 36 x_2 - a x_6 + 9 a x_7), \\
  y_5 &=  3 Q^9 (18 b x_2 - 2 a^2 (x_3 + 2 x_5) + 3 a b x_6 + 6 a b x_7), \\
  y_6 &=  8 Q (4 a^3+27 b^2) (6 x_1 - 6 x_2 - a x_6), \\
  y_7 &=  -4 (4 a^3+27 b^2) (3 x_1 - a x_6 - a x_7).
\end{align*}
This transformation has determinant
$2^{26} 3^8 a^8 Q^{38} (4 a^3 + 27 b^2)^4$.  Dividing $\QQ$ and $\FF$
by $2^{18} 3^4 a^4 (4 a^3 + 27 b^2)^2$ and
$2^{32} 3^8 a^8 (4 a^3 + 27 b^2)^3$, and eliminating $Q$ by the rule
$Q^{13} = 16(4 a^3 + 27b^2)$, we obtain
\begin{align*}
\QQ(x_1,  \ldots, x_7) = x_1 x_7 + x_2 x_6 + x_3 x_5 + x_4^2,
\end{align*}
and 
\begin{align*}
  \FF(&x_1,  \ldots, x_7) = 
        48 a x_1^3 x_2 + 36 b x_1^3 x_3 + 72 b x_1^3 x_5 + 8 a^2 x_1^3 x_6 + 16 a^2 x_1^3 x_7 \\ 
      & - 144 a x_1^2 x_2^2 - 216 b x_1^2 x_2 x_3 + 432 b x_1^2 x_2 x_4 + 432 b x_1^2 x_2 x_5 - 
        48 a^2 x_1^2 x_2 x_7 \\
      & + \ldots  + 12a^2(a^3 + 9 b^2) x_6^2 x_7^2 + 56a^2(a^3 + 7 b^2) x_6 x_7^3 
        + 16a^2(3 a^3 + 20 b^2) x_7^4.
\end{align*}
These polynomials $\QQ$ and $\FF$ have weights $2$ and $10$ with
respect to the grading in Remark~\ref{rem:gradings}.

For $f$ and $g$ homogeneous polynomials in $x_1, \ldots,x_7$ we define
  \[ \langle f,g \rangle = \trace ( H(f) H(\QQ)^{-1} H(g) H(\QQ)^{-1} ).\]
\begin{Theorem}
\label{thm:XE13-skew}
Let $E/K$ be the elliptic curve $y^2 = x^3 +a x+ b$.  Let
$f_1,\ldots,f_7$ be a basis for the space of cubic forms $f$
satisfying $\langle f, \FF - 48(4a^3 + 27b^2) \QQ^2 \rangle = 0$.
Then
\[ X_E(13,2) \isom \{ f_1 = \ldots = f_7 = \FF + 16 (4a^3 + 27b^2)
\QQ^2 = 0 \} \subset \PP^6.\]
\end{Theorem}
\begin{proof}
The proof is similar to that of Theorem~\ref{thm:XE13-dir},
except that we now form the matrix $h$ by evaluating the skew
covariants
\begin{equation*}
Q^{11} \ww_4,\, Q^{4} \ww_5, \, Q^{10} \ww_6, \, Q^{3} \ww_7, \, Q^{9} \ww_8, \, Q \ww_{11}, \, \ww_{13} 
\end{equation*}
and show using the definition of a skew covariant
(see Definition~\ref{def:skewcov}) that
\begin{equation*}
\sigma(h) h^{-1} \propto \widetilde{\rho (\sigma(\phi) \phi^{-1})} 
\end{equation*}
for all $\sigma \in \Gal(\Kbar/K)$. 
\end{proof}

Again the cubic forms $f_1, \ldots, f_7$, as polynomials with coefficients in 
$\Z[a,b]$, are available from~\cite{magmafile}. Alternatively, they may
be computed from the description given in Section~\ref{sec:state-curves}.

Next we compute the $j$-map $X_E(13,2) \to \PP^1$.  Revisiting
Lemma~\ref{lem:N} with $\QQ$ and $\FF$ in place of $Q$ and $F$ we
obtain a skew symmetric matrix $\NN$ given by
$\NN_{ij} = (\frac{\partial}{\partial x_i} \wedge
\frac{\partial}{\partial x_j}) \Phi$ where
\begin{equation*}
 \Phi = (x_1 \wedge x_4 \wedge x_7) - (x_1 \wedge x_5 \wedge x_6) - (x_2 \wedge x_3 \wedge x_7) - (x_2 \wedge x_4 \wedge x_6) - (x_3 \wedge x_4 \wedge x_5). 
\end{equation*}
We put $\vvv_3 = \nabla_{\!\QQ}\, \FF$,
$\vvv_4 = H(\QQ)^{-1} \NN \vvv_3$, $\vvv_{9} = \vvv_3 \circ \vvv_3$
and $\cc_6 = \coeff ( \QQ( \vvv_4 + t \vvv_9) , t)$.  The map
$j : X_E(13,2) \to \PP^1$ satisfies $j - 1728 = \xi \cc_6^2/\QQ^{13}$
for some constant $\xi$. In principle we could compute $\xi$ by
carefully keeping track of all the changes of coordinates and
rescalings described above, but in practice it is simpler to look at
some numerical examples.  We find that
\[ j = 1728 - \frac{\cc_6^2}{2^{40}(4 a^3 + 27 b^2)^{10} \QQ^{13}}. \]

\begin{Remark}
  We have arranged that the forms $\QQ$ and $\Phi$ do not depend on
  $a$ and $b$, and indeed, up to a relabelling of the variables, are
  the same as the forms we started with. Therefore, exactly as in
  Remark~\ref{rem:g2}, we have twisted by a cocycle taking values in
  $G_2$.
\end{Remark}

\subsection{The cases $j(E) = 0, 1728$}
\label{sec:j=0,1728} 
We have shown that the equations for $X_E(13,1)$ and $X_E(13,2)$ in
Theorems~\ref{thm:XE13-dir} and~\ref{thm:XE13-skew} are correct for
all elliptic curves $E$ with $j(E) \not= 0, 1728$. 
We now remove this restriction. The first step is to show that if the
theorems are correct for some elliptic curve $E$ then they are correct
for any $2$-isogenous elliptic curve $F$.

Let $E$ be the elliptic curve $y^2 = x^3 + ax+b$, and let $F$ be the
elliptic curve $y^2 = x^3 + A x + B$ where $A = -15 \ta^2 - 4a$,
$B = 14 a \ta + 22 b$ and $\ta$ is a root of $x^3 + ax + b =0$.  If we
put $c = 3\ta$ and $d = 3 \ta^2 + a$ then $E$ is isomorphic to
$y^2 = x(x^2 + cx + d)$ and $F$ is isomorphic to
$y^2 = x((x - c)^2 - 4d)$.  In particular $E$ and $F$ are
$2$-isogenous.

Starting from the equations in Theorems~\ref{thm:XE13-dir}
and~\ref{thm:XE13-skew}, we find there is an isomorphism
$X_E(13,1) \isom X_F(13,2)$ given by
$(x_1 : \ldots:x_7) \mapsto (x'_1 : \ldots:x'_7)$ where
\begin{align*}
 x'_1 &=   (14 \ta^2 + 8 a) x_1 + 32 \ta^2 x_2 - 
  (3 a \ta - 25 b) x_3 + (38 a \ta - 18 b) x_4 - (18 a \ta^2 \\ & + 
   30 b \ta + 16 a^2) x_5 - (11 a \ta^2 - 3 b \ta + 24 a^2) x_6 
    + (144 b \ta^2 - 44 a^2 \ta + 132 a b) x_7, \\
 x'_2 &=   (14 \ta^2 + 8 a) x_2 + (3 a \ta + 7 b) x_3
    + (11 a \ta + 15 b) x_4 - (17 a \ta^2 - 9 b \ta + 8 a^2) x_5 \\ & - 
           (5 a \ta^2 - 21 b \ta + 8 a^2) x_6 + (144 b \ta^2 
     + 10 a^2 \ta + 66 a b) x_7, \\
 x'_3 &=   8 \ta x_1 + 8 \ta x_2 + 8 \ta^2 x_3 
   + (30 \ta^2 + 24 a) x_4 + (10 a \ta - 6 b) x_5 
   - (8 a \ta + 24 b) x_6 \\ & +  (12 a \ta^2 - 108 b \ta) x_7, \\
 x'_4 &=   4 \ta x_1 + 8 \ta x_2 + (5 \ta^2 + 4 a) x_3 
   - (5 a \ta - 3 b) x_6, \\
 x'_5 &=   -4 \ta x_1 - 12 \ta x_2 + (3 \ta^2 - 4 a) x_4 
   + (a \ta + 9 b) x_5 + (6 a \ta^2 - 18 b \ta + 16 a^2) x_7, \\
 x'_6 &=   -4 x_1 - 12 x_2 + 6 \ta x_4 - (6 \ta^2 + 4 a) x_5 
   - 12 a \ta x_7, \\
 x'_7 &=   2 x_1 + 4 x_2 + \ta x_3 + (3 \ta^2 + 2 a) x_6.
\end{align*}
The determinant of this transformation is
$-2^{10} 3^2 d^3 (c^2 - 4 d)^5$, and so in particular is
non-zero. 

Let $E_{a,b}$ be the elliptic curve $y^2 = x^3 + ax + b$. Since
$E_{1,0}$ is $2$-power isogenous to $y^2 = x^3 - 44 x + 112$ and
$E_{0,1}$ is $2$-isogenous to $y^2 = x^3 - 15 x + 22$, it follows that
Theorems~\ref{thm:XE13-dir} and~\ref{thm:XE13-skew} hold for the
elliptic curves $E_{1,0}$ and $E_{0,1}$. It remains to show that if
these results hold for some elliptic curve with $j = 0,1728$ then they
hold for all such curves.

The non-cuspidal points of $X_E(p,k)$ correspond to pairs $(F,\psi)$,
where $F$ is an elliptic curve and $\psi : F[p] \to E[p]$ is a
isomorphism that raises the Weil pairing to the power $k$.  We write
$\SL(E[p])$ for the group of automorphisms of $E[p]$ that respect the
Weil pairing.  Then $\SL(E[p])$ acts on $X_E(p,k)$ via
$\gamma : (F,\psi) \mapsto (F,\gamma \psi)$. There is therefore a
group homomorphism
\begin{equation}
\label{def:piE}
\pi_E : \Aut(E)/\{\pm 1\} \to \SL(E[p])/\{\pm 1\} \to \Aut(X_E(p,k)). 
\end{equation}

\begin{Lemma} 
\label{lem:alphabeta}
Let $E$ and $E'$ be elliptic curves defined over $K$ and
$\alpha : E' \to E$ an isomorphism defined over $\Kbar$.  Then there
is an isomorphism $\beta : X_{E'}(p,k) \to X_E(p,k)$ defined over
$\Kbar$ satisfying
\[  \sigma(\beta) \beta^{-1} = \pi_E ( \sigma(\alpha) \alpha^{-1} ) \]
for all $\sigma \in \Gal(\Kbar/K)$. 
\end{Lemma}

\begin{proof}
  Let $\beta : X_{E'}(p,k) \to X_E(p,k)$ be the isomorphism given on
  non-cuspidal points by $(F,\psi) \mapsto (F,\alpha \psi)$. Then
  $\sigma(\beta) \beta^{-1}$ maps
  $(F,\psi) \mapsto (F, \sigma(\alpha) \alpha^{-1} \psi)$, and is
  therefore equal to $\pi_E ( \sigma(\alpha) \alpha^{-1} )$.
\end{proof}

\noindent
{\em Proof of Theorem~\ref{thm:XE13-dir} for $j=1728$.}
We have already shown that the theorem holds for $E = E_{1,0}$.  We
now prove it for $E' = E_{a,0}$. We identify $\Aut(E) = \mu_4$ via
$\zeta : (x,y) \mapsto (\zeta^{-2}x, \zeta^{-3}y)$. Let
$\alpha : E' \to E$ be the isomorphism given by
$(x,y) \mapsto (a^{-1/2} x, a^{-3/4} y)$. Then
\begin{equation}
\label{twisteq1}
\sigma(\alpha) \alpha^{-1} = \frac{\sigma(a^{1/4})}{a^{1/4}}.
\end{equation}
Let $X,X_a \subset \PP^6$ be the models claimed for $X_{E}(13,1)$ and
$X_{E'}(13,1)$ in Theorem~\ref{thm:XE13-dir}. We have already shown
that $X \isom X_E(13,1)$.  From the grading in
Remark~\ref{rem:gradings} we construct an isomorphism
$\beta : X_a \to X$ given by
\[ (x_1: \ldots : x_7) \mapsto (x_1 : x_2 : a^{1/2} x_3 : a^{1/2} x_4 : a x_5 : 
  a x_6 : a^{3/2} x_7 ). \]
Then 
\begin{equation}
\label{twisteq2}
 \sigma(\beta) \beta^{-1} = \left\{ 
\begin{array}{ll}
 1 & \text{ if } \sigma(a^{1/2}) = a^{1/2} \\
 \iota & \text{ if } \sigma(a^{1/2}) = -a^{1/2} 
\end{array} \right. 
\end{equation}
where
\[ \iota : (x_1: \ldots : x_7) \mapsto (x_1 : x_2 : -x_3 : -x_4 : x_5 : x_6 : -x_7 ).  \]

If $\sigma(\beta) \beta^{-1} = \pi_E( \sigma(\alpha) \alpha^{-1} )$
for all $\sigma \in \Gal(\Kbar/K)$ then we see by
Lemma~\ref{lem:alphabeta} that $X_a$ and $X_{E'}(13,1)$ are twists of
$X = X_E(13,1)$ by the same cocycle, and are therefore isomorphic over
$K$. Comparing~\eqref{twisteq1} and~\eqref{twisteq2}, it remains to
show that $\pi_E$ sends $\zeta_4 \mapsto \iota$. More generally, we
claim that $\iota$ is the unique involution of $X_E(13,1)$ defined
over $\Q(i)$.  By \cite[Theorem~20.40]{AR} the second map in
\eqref{def:piE} is an isomorphism. This reduces our claim to one about
$\SL(E[13])/ \{ \pm 1 \}$.  It may be checked, for example by
consulting the LMFDB \cite{lmfdb}, that the mod $13$ Galois
representation attached to $E/\Q(i)$ has image a split Cartan
subgroup, i.e., the subgroup $C$ of diagonal matrices in
$\GL_2(\Z/13\Z)$. But then the group
\[  \{ h \in \SL_2(\Z/13\Z) : g h g^{-1} = \pm h \text{ for all } g \in C \}
/ \{ \pm 1 \} \]
is cyclic of order $6$, and so contains a unique element of order $2$.
\qed

\medskip

\noindent
{\em Proof of Theorem~\ref{thm:XE13-dir} for $j=0$.}
We have already shown that the theorem holds for $E = E_{0,1}$.  We
now prove it for $E' = E_{0,b}$. We identify $\Aut(E) = \mu_6$ via
$\zeta : (x,y) \mapsto (\zeta^{-2}x, \zeta^{-3}y)$. Let
$\alpha : E' \to E$ be the isomorphism given by
$(x,y) \mapsto (b^{-1/3} x, b^{-2/3} y)$. Then
\begin{equation*}
\sigma(\alpha) \alpha^{-1} = \frac{\sigma(b^{1/6})}{b^{1/6}}. 
\end{equation*}
Let $X, X_b \subset \PP^6$ be the models claimed for $X_{E}(13,1)$ and
$X_{E'}(13,1)$ in Theorem~\ref{thm:XE13-dir}. We have already shown
that $X \isom X_E(13,1)$.  From the grading in
Remark~\ref{rem:gradings} we construct an isomorphism
$\beta : X_b \to X$ given by
\[ (x_1: \ldots : x_7) \mapsto ( x_1 : x_2 : b^{1/3} x_3 : b^{1/3} x_4 : b^{2/3} x_5 :
  b^{2/3} x_6 : b x_7 ). \]
Then 
\begin{equation*}
 \sigma(\beta) \beta^{-1} = \left\{ 
\begin{array}{ll}
 1 & \text{ if } \sigma(b^{1/3}) = b^{1/3} \\
 \eps & \text{ if } \sigma(b^{1/3}) = \zeta_3 b^{1/3} \\
 \eps^2 & \text{ if } \sigma(b^{1/3}) = \zeta_3^2 b^{1/3} 
\end{array} \right. 
\end{equation*}
where
\[ \eps : (x_1: \ldots : x_7) \mapsto ( x_1 : x_2 : \zeta_3 x_3 : 
  \zeta_3 x_4 : \zeta_3^2 x_5 : \zeta_3^2 x_6 : x_7 ). \]

  Arguing as in the proof with $j=1728$, it remains to show that the
  map $\pi_E$ sends $\zeta_6 \mapsto \eps$.  We find that $\eps$ and
  $\eps^2$ are the only order $3$ automorphisms of $X_E(13,1)$ defined
  over $\Q(\zeta_3)$. Therefore $X_{E'}(13,1)$ is isomorphic to $X_b$
  or $X_{1/b}$.  To rule out the latter we take $b = 2$ and consider
  the $3$-isogenous elliptic curves $E' : y^2 = x^3 + 2$ and
  $F : y^2 = x^3 - 120 x + 506$. Since $3$ is a quadratic residue
  mod~$13$ we have $X_{E'}(13,1) \isom X_{F}(13,1)$.  However the
  curves $X_{1/2}$ and $X_{F}(13,1)$ are not isomorphic, since they
  have a different number of points mod~$19$.  \qed

\medskip

The proof of Theorem~\ref{thm:XE13-skew} in the cases $j=0,1728$ is
similar.

\section{Modular diagonal quotient surfaces}
\label{sec:mdqs}

In this section we prove Theorem~\ref{thm:Zeqns}.

\subsection{Equations for $Z(13,1)$}
\label{sec:Z(13,1)}
Let $X = X(13) \subset \PP^6$ be the $A$-curve as defined in
Section~\ref{sec:X13}.  By Lemma~\ref{lem:quot} the surface $Z(13,1)$
is birational to the quotient of
$X \times X \subset \PP^6 \times \PP^6$ by the diagonal action of
$G \isom \PSL_2(\Z/13\Z)$.  We write $x_0, \ldots, x_6$ and
$y_0, \ldots, y_6$ for our coordinates on the first and second copies
of $\PP^6$.

\begin{Definition}
  A {\em bi-invariant} of degree $(m,n)$ is a polynomial in
  $x_0, \ldots, x_6$ and $y_0,\ldots,y_6$, that is homogeneous of
  degrees $m$ and $n$ in the two sets of variables, and is invariant
  under the diagonal action of $G$.
\end{Definition}

In principle, we may obtain equations for $Z(13,1)$ by computing
generators and relations for the ring of bi-invariants mod
$I(X \times X)$.  In practice we only compute some of the generators
and some of the relations, and then explain why these are sufficient.

At the start of Section~\ref{sec:X13} we defined invariants $Q$ and
$F$ of degrees $2$ and $4$. We write $Q_{20}$ and $F_{40}$ for these
same polynomials viewed as bi-invariants of degrees $(2,0)$ and
$(4,0)$. More generally we define bi-invariants $Q_{ij}$ and $F_{ij}$
by the rules
\begin{align*}
  Q(\lambda x_0 + \mu y_0, \ldots , \lambda x_6 + \mu y_6)
    & = \lambda^2 Q_{20} + \lambda \mu Q_{11} + \mu^2 Q_{02} \\ 
  F(\lambda x_0 + \mu y_0, \ldots , \lambda x_6 + \mu y_6)
    & = \lambda^4 F_{40} + \lambda^3 \mu F_{31} + \ldots + \mu^4 F_{04} 
\end{align*}
where the subscripts indicate the degree. 

The dimension of the space of bi-invariants of degree $(m,n)$ may be
computed from the character table for $G$. For some small values of
$m$ and $n$ these dimensions are as follows.
\[ \begin{array}{c|ccccccccccccc}
   & 0 & 1 & 2 & 3 & 4 & 5 & 6 & 7 & 8 & 9 & 10 \\ \hline
0 & 1 & 0 & 1 & 0 & 2 & 0 & 4 & 1 & 7 & 3 & 14 \\
1 & 0 & 1 & 0 & 2 & 1 & 5 & 5 & 14 & 17 & 37 & 48 \\
2 & 1 & 0 & 3 & 1 & 10 & 9 & 32 & 38 & 90 & 118 & 226 \\
3 & 0 & 2 & 1 & 10 & 14 & 41 & 67 & 142 & 222 & 402 & 602 \\
4 & 2 & 1 & 10 & 14 & 51 & 82 & 198 & 316 & 610 & 938 & 1592 \\
5 & 0 & 5 & 9 & 41 & 82 & 206 & 377 & 746 & 1244 & 2152 & 3346 
\end{array} \]
To compute the bi-invariants of a given degree we use
the efficient averaging method described in Remark~\ref{rem:effavg}.

To find relations between the bi-invariants modulo $I(X \times X)$ we
initially worked mod $p$ for some moderately sized prime $p$,
employing the heuristic that a polynomial vanishing at many
$\F_p$-points on $X \times X$ is likely to vanish on the whole
surface. One way to establish these relations rigorously would be to
employ the Gr\"obner basis machinery in Magma. However this proved too
slow in all but the simplest cases. We instead used the following
lemma, which is an easy consequence of Bezout's theorem.
\begin{Lemma}
\label{lem:23}
Let $I$ be a bihomogeneous form of degree $(m,n)$ with $m,n \le 23$.
If $I$ vanishes at all points $(P,Q) \in X \times X$ with
$j(P),j(Q) \in \{0,1728,\infty\}$ then $I$ vanishes on $X \times X$.
\end{Lemma}
\begin{proof}
  We fix $P_0 \in X$ with $j(P_0) \in \{0,1728,\infty\}$, and let
  $f(Q) = I(P_0,Q)$. The hypersurface $\{f = 0\} \subset \PP^6$ meets
  $X$ in at least $84 + 364 + 546 > 42 \times 23$ points. Since $X$
  has degree $42$ and $f$ has degree $n \le 23$ it follows by Bezout's
  theorem that $f$ vanishes identically on $X$. Therefore $I$ vanishes
  on $\{P_0\} \times X$.  We now fix any $Q_0 \in X$. Applying the
  same argument to $g(P) = I(P,Q_0)$, and using that $m \le 23$, shows
  that $I$ vanishes on $X \times X$.
\end{proof}
We note that if the bihomogeneous form in Lemma~\ref{lem:23} is a
bi-invariant (or a skew bi-invariant, as defined in the next section)
then this significantly reduces the amount of work needed to check the
hypotheses of the lemma.

If $I$ is a bi-invariant of degree $(m,n)$ then we write $I'$ for the
bi-invariant of degree $(n,m)$ obtained by switching the $x$'s and
$y$'s. A bi-invariant of degree $(m,m)$ is {\em symmetric} if
$I' = I$, and {\em anti-symmetric} if $I' = -I$.

The vector space of bi-invariants of degree $(3,3)$ has dimension
$10$, and the subspace of symmetric bi-invariants has dimension $9$.
Making a good choice of basis for this space significantly simplifies
the calculations that follow. To specify our choice of basis
$z_1, \ldots, z_9$, we let $m_1, \ldots, m_{10}$ be the monomials
\begin{align*}
  x_2^3 y_0^2 y_1,  \,\,  x_1 x_4^2 y_0^2 y_1,  \,\,  x_2^2 x_3 y_0 y_1^2, & \,\, 
    x_3 x_4 x_5 y_0 y_1^2, \,\,  x_2 x_3^2 y_1^3, \\ 
    \,\,&   x_4^3 y_1^3, \,\,  x_0^2 x_6 y_1^3, \,\, 
    x_1 x_4 x_6 y_1^3, \,\,  x_2 x_5 x_6 y_1^3, \,\,  x_3 x_6^2 y_1^3, 
\end{align*}
and then record the coefficients of these monomials in a table.
\[
\begin{array}{c|cccccccccc}
& m_1 & m_2 & m_3 & m_4 & m_5 & m_6 & m_7 & m_8 & m_9 & m_{10} \\ \hline
z_1 & 1 &-2 &-2 & 4 & 0 & 1 & 1 & 1 &-1 & 0 \\
z_2 &-2 & 0 & 4 &-2 &-1 &-1 &-2 &-1 & 1 & 0 \\
z_3 & 0 &-3 & 0 & 1 &-1 & 0 & 0 & 0 & 0 & 0 \\
z_4 & 0 & 0 &-2 & 0 & 0 & 0 & 0 & 0 &-1 & 0 \\
z_5 & 1 & 0 &-3 & 0 & 0 & 0 & 1 & 0 &-1 & 0 \\
z_6 & 1 &-1 &-3 & 0 & 0 & 0 & 1 & 0 &-1 & 0 \\
z_7 & 1 &-1 &-2 & 4 & 0 & 0 & 1 & 1 &-1 & 0 \\
z_8 & 0 &-4 & 2 & 0 &-1 & 0 & 0 & 0 & 1 &-1 \\
z_9 & 0 & 0 & 1 & 2 & 0 & 0 & 0 & 0 & 0 & 0 
\end{array}
\]
Some of these bi-invariants may also be described in terms of the
$Q_{ij}$ and $F_{ij}$. Specifically we have
\begin{align}
\label{eqn1}
 Q_{11}^3 &=  z_1 + z_5 - z_6 - z_7, \\
\label{eqn2}
 Q_{11} Q_{20} Q_{02} &= z_5 - z_6, \\
\nonumber
 Q_{11} F_{22} &= -3 (z_6 - z_7 + z_9), \\
\nonumber
 Q_{20} F_{13} + Q_{02} F_{31} &= 
z_1 + z_2 - 4 z_4  + z_6  + z_7 - z_8 - 3 z_9.
\end{align}

We find using Lemma~\ref{lem:23} that $z_9$ and the following 9
quadratic forms in $z_1, \ldots, z_8$ vanish identically on
$X \times X$.
\begin{align*}
  &z_1 z_4 - z_3 z_5, & &z_1 z_7 - z_1 z_8 + z_2 z_7 - z_4 z_6 + z_5 z_7,  \\
  &z_1 z_6 - z_3 z_7, & &z_1 (z_1 + z_2 - z_3 - z_4 + z_5 - z_7) - z_2z_6 + z_3 z_6,
 \\
  &z_4 z_7 - z_5 z_6, & &z_4 (z_1 + z_2 - z_3 - z_4 + z_5) - z_3 z_8, \\
  &z_1^2 + z_1 z_2 - z_2 z_4, & &z_5 (z_1 + z_2 - z_3 - z_4 + z_5) - z_1 z_8, \\
  &&  &z_8 (z_1 + z_5 - z_6 - z_7) - z_5 z_7.
\end{align*}
These quadratic forms define a rational surface $\Sigma \subset \PP^7$. 
Indeed, the map $\Sigma \to \PP^2$ given by projection onto
the first 3 coordinates, is a birational map with inverse
\begin{equation}
\label{param1}
\begin{aligned}
(z_1 , \ldots , z_8)  = ( r,  s, &\, 1,   r (r + s)/s, 
    r^2(r + s)/s, r f(r,s)/(s(r^2 + s - 1)), \\
     &  r^2 f(r,s)/(s(r^2 + s - 1)), r(r + s) f(r,s)/s^2 ) 
\end{aligned}
\end{equation}
where $f(r,s) = r^3 + r^2 s - r^2 + s^2 - s$.

The space of anti-symmetric bi-invariants of degree $(3,3)$ is
1-dimensional, spanned by $w = Q_{20} F_{13} - Q_{02} F_{31}$.  We
write $Z(13,1)$ as a double cover of $\Sigma$ by finding an expression
for $w^2$ in terms of $z_1, \ldots, z_8$.  Specifically, working mod
$I(X \times X)$, we find the relation
$w^2 + 64(Q_{20} Q_{02} )^3 = g(z_1, \ldots, z_8)$ where
\begin{align*}
g(&z_1, \ldots, z_8) =
  z_1 z_2 + z_2^2 - 48 z_2 z_3 + 48 z_2 z_5 + 126 z_2 z_6 
  - 48 z_2 z_7 - 2 z_2 z_8 \\ & + 48 z_3^2 - 7 z_3 z_4 
  - 57 z_3 z_5 - 108 z_3 z_6 + 126 z_3 z_7 + 21 z_3 z_8 
  + z_4^2 - 41 z_4 z_5 \\ & - 26 z_4 z_6 + 8 z_4 z_8 
  + 104 z_5^2 - 60 z_5 z_6 - 106 z_5 z_7 + 120 z_5 z_8 
  + 37 z_6^2 - 10 z_6 z_7 \\ & - 158 z_6 z_8 + z_7^2 
  - 70 z_7 z_8 + z_8^2.
\end{align*}
It follows by \eqref{eqn1}, \eqref{eqn2} and \eqref{param1} that
\[ w^2 = ((r - 1)/(s^2(r^2 + s - 1)))^2 F_1(r,s,1) \]
where $F_1$ is the polynomial defined in
Remark~\ref{completethesquare}.  The bi-invariants therefore define a
rational map from $Z(13,1)$ to the surface $y^2 = F_1(r,s,1)$.  We
show in Remark~\ref{deg1-dir} below that this map has degree 1.

We now compute the maps $j$ and $j'$ giving the moduli interpretation
of $Z(13,1)$. To do this we need some more bi-invariants, and some
more relations. If $\vv$ is a covariant of degree $m$ (see
Definition~\ref{def:cov}) and $\yy = (y_0, \ldots, y_6)^T$ then
$\yy^T H(Q) \vv$ is a bi-invariant of degree $(m,1)$. Applying this
construction to $\vv_4$ as defined in Section~\ref{sec:XE(13,1)} gives
a bi-invariant $I_{41}$. We put $I_{32} = 
(\sum y_i \frac{\partial}{\partial x_i}) I_{41}$
and $I_{23} = I'_{32}$.  Let $c_6$ be the invariant of degree 13
defined at the end of Section~\ref{sec:X13}, now viewed as a
bi-invariant of degree $(13,0)$. Then $c'_6$ has degree $(0,13)$.  Let
$\alpha = Q_{20} I_{23}^2$, $\alpha' = Q_{02} I_{32}^2$,
$\beta = Q_{02}^6 I_{23} c_6$, $\beta' = Q_{20}^6 I_{32} c'_6$ and
$\gamma = (Q_{20} Q_{02})^3$.  Working mod $I(X \times X)$ we find
some relations
\begin{align*}
f_1 (\alpha + \alpha') &= f_3, & Q_{11} I_{32} I_{23} &= g_2, \\
f_2 (\beta + \beta') &= f_7, &
 h_2 c_6 c'_6 &= Q_{11} (\ell_6 + \ell_4 \gamma + \ell_2 \gamma^2 - 64 \gamma^3),
\end{align*}
where each $f_i, g_i, h_i, \ell_i$ is a homogeneous polynomial of
degree $i$ in $z_1, \ldots, z_8$.  These polynomials are available
from~\cite{magmafile}.  The relations were checked using
Lemma~\ref{lem:23}.  Using \eqref{eqn1}, \eqref{eqn2} and
\eqref{param1} we may then write the coefficients of the quadratics
$(Y - \alpha)(Y - \alpha')$ and $(Y - \beta)(Y - \beta')$ as rational
functions in $r$ and $s$. The discriminant of each quadratic is equal
to $F_1(r,s,1)$ times a square.  Moreover, by Theorem~\ref{myinv} we
have $j = 1728 - \beta^2/(\alpha \gamma^4)$, which we may then write
as an element of $\Q(r,s,\sqrt{F_1(r,s,1)})$.  The final expressions
for $j$ and $j'$ are too complicated to record here, but take the form
specified in Theorem~\ref{thm:Zeqns}, and are available
from~\cite{magmafile}.

\begin{Remark}
\label{deg1-dir}
We have constructed rational maps
\begin{equation}
\label{ratmaps}
X \times X \ra Z(13,1) \ra \{y^2 = F_1(r,s,1)\} \stackrel{(j,j')}{\relbar\joinrel\longrightarrow} 
\PP^1 \times \PP^1. 
\end{equation}
The composite corresponds to a Galois extension of function fields,
with Galois group $G \times G$. Since $G \isom \PSL_2(\Z/13\Z)$ is a
simple group, the diagonal subgroup $\Delta_G \subset G \times G$ is a
maximal subgroup.  Therefore one of the last two maps
in~\eqref{ratmaps} is birational.  However if the last map were
birational, then this would mean that in attempting to quotient out by
$\Delta_G$, we had in fact quotiented out by $G \times G$. To exclude
this possibility we may check, for example, that the rational function
$Q_{11}^2/(Q_{20} Q_{02})$ on $X \times X$ is not
$G \times G$-invariant.  In fact, it is not even
$\langle M_6 \rangle \times \langle M_6 \rangle$-invariant.  Therefore
$Z(13,1)$ is birational to $\{y^2 = F_1(r,s,1)\}$, and this completes
the proof of Theorem~\ref{thm:Zeqns} in the case $k=1$.
\end{Remark}

\subsection{Equations for $Z(13,2)$}
\label{sec:Z(13,2)}
The calculations here are similar to those in the last section.  The
main difference is that we modify the definition of a bi-invariant. As
in Section~\ref{sec:XE(13,2)} we write $g \mapsto \widetilde{g}$ for
the automorphism of $G$ induced by $\zeta \mapsto \zeta^2$.

\begin{Definition}
  A {\em skew bi-invariant} of degree $(m,n)$ is a polynomial in
  $x_0, \ldots, x_6$ and $y_0,\ldots,y_6$, that is homogeneous of
  degrees $m$ and $n$ in the two sets of variables, and is invariant
  under the action of $G$ via
  $g : (x,y) \mapsto (g x, \widetilde{g}y)$.
\end{Definition}

The polynomials $Q_{20}$ and $Q_{02}$ defined in Section~\ref{sec:Z(13,1)}
are skew bi-invariants, but $Q_{11}$ is not. The
dimension of
the space of skew bi-invariants of degree $(m,n)$ may again
be computed from the character table for $G$. For some small
values of $m$ and $n$ these dimensions are as follows. 
\[ \begin{array}{c|ccccccccccc}
   & 0 & 1 & 2 & 3 & 4 & 5 & 6 & 7 & 8 & 9 & 10 \\ \hline
0 & 1 & 0 & 1 & 0 & 2 & 0 & 4 & 1 & 7 & 3 & 14 \\
1 & 0 & 0 & 0 & 1 & 1 & 4 & 5 & 14 & 17 & 37 & 48 \\
2 & 1 & 0 & 3 & 1 & 10 & 9 & 32 & 38 & 90 & 118 & 226 \\
3 & 0 & 1 & 1 & 9 & 14 & 40 & 67 & 142 & 222 & 402 & 602 \\
4 & 2 & 1 & 10 & 14 & 51 & 82 & 198 & 316 & 610 & 938 & 1592 \\
5 & 0 & 4 & 9 & 40 & 82 & 205 & 377 & 746 & 1244 & 2152 & 3346 
\end{array} \]

In particular the spaces of skew bi-invariants of degrees $(2,2)$ and
$(3,3)$ have dimensions $3$ and $9$.  The first of these spaces has
basis $t_1,t_2,t_3$ where
\begin{align*}
t_1 &=   2 x_0^2 y_0^2 + x_4^2 y_0 y_1 + \ldots \\ 
t_2 &=   (x_0^2 - x_1 x_4 - x_2 x_5 - x_3 x_6) y_0^2 
               + x_4^2 y_0 y_1 + \ldots \\ 
t_3 &=  (-5 x_0^2 + x_1 x_4 + x_2 x_5 + x_3 x_6) y_0^2 - ( x_4^2 + 
        2 x_1 x_6) y_0 y_1 + \ldots. 
\end{align*}
To specify the first $5$ polynomials in the basis $u_1, \ldots,u_9$ we
chose for the space of skew bi-invariants of degree $(3,3)$, we let
$m_1, \ldots, m_9$ be the monomials
\[
    x_2^2 x_3 y_0^2 y_1, \,\,
    x_1 x_2^2 y_0 y_1^2, \,\,
    x_3^3 y_0 y_1^2, \,\,
    x_0^2 x_5 y_0 y_1^2, \,\,
    x_1 x_3^2 y_1^3, \,\,
    x_2^2 x_4 y_1^3, \,\,
    x_4^2 x_5 y_1^3, \,\,
    x_1 x_5 x_6 y_1^3, \,\,
    x_0 x_6^2 y_1^3,
\]
and then record the coefficients of these monomials in a table
\[ \begin{array}{c|ccccccccc}
& m_1 & m_2 & m_3 & m_4 & m_5 & m_6 & m_7 & m_8 & m_9 \\ \hline
u_1 &  0 &-2 & 0 & 0 &-1 & 0 & 0 &-1 & 0 \\
u_2 & -3 & 1 & 0 &-1 & 0 & 0 & 0 & 1 & 0 \\
u_3 &  0 & 2 & 0 & 0 & 1 & 0 & 1 & 2 & 0 \\
u_4 & -1 & 1 & 0 & 1 & 0 & 0 & 0 & 0 & 0 \\
u_5 & -1 &-2 &-1 &-3 & 0 & 0 & 0 & 0 & 1
\end{array} \]
Amongst other relations, we found that $u_6, \ldots, u_9$ and 
the following polynomials
vanish identically on $X \times X$.
\begin{align*}
& u_3u_4 - u_1u_5, &&
u_2^3 - u_1u_2u_3 - 2u_2^2u_3 - u_2^2u_4 + u_1u_3u_4 + u_2u_3u_4 - u_2u_3u_5, \\
& t_2 u_4 - t_3 u_1, && 
t_1 u_1 (u_2 - u_3) - t_2 (u_1 u_2 - u_1 u_4 + u_2 u_3 + u_2 u_5), \\
& t_1 t_2^2 - u_1 u_3.
\end{align*}
The first two relations define a rational surface
$\Sigma \subset \PP^4$, parametrised by
\[ (u_1 , \ldots , u_5) = (r , 1 , -r s , f(r,s) , -s f(r,s)) \]
where
$f(r,s) = (r^2 s + 2 r s + 1)/(r^2 s + r s^2 + r s + 1)$.
The other three show that
\[  (t_1,t_2,t_3)  = (- s (r^2 + r s + r + 1)/\tau,
    r (r^2 s + r s^2 + r s + 1)/\tau,
    (r^2 s + 2 r s + 1)/\tau ) \] 
where 
$\tau^3 = (r^2 + r s + r + 1) (r^2 s + r s^2 + r s + 1)^2$.

For $I$ a skew bi-invariant we write $I'$ for the skew bi-invariant
obtained as 
\[ I'(x;y) = I(y; -x_0,-x_2,-x_3,-x_4,-x_5,-x_6,-x_1). \]
It may be checked that $I'' = I$.  The space of skew bi-invariants of
degree $(3,2)$ is spanned by
$I_{32} = 12 (x_1 x_3 x_5 - x_2 x_4 x_6) y_0^2 + \ldots$ and the space
of skew bi-invariants of degree $(4,1)$ is spanned by
$I_{41} = \yy^T H(Q) \ww_4$ where $\ww_4$ is the skew covariant
defined in Section~\ref{sec:XE(13,2)}.  We put $I_{23} = I'_{32}$ and
$I_{14} = I_{41}'$.
Let $\alpha = Q_{20}^3 I_{14}^2$, $\alpha' = Q_{02}^3 I_{41}^2$,
$\beta = Q_{02}^5 I_{14} c_6$, $\beta' = Q_{20}^5 I_{41} c_6'$ and
$\gamma = I_{32} I_{23}$. Working mod $I(X \times X)$ we find some
relations
\begin{align*}
        Q_{20} Q_{02} &= t_1 - t_2, \\
        I_{41} I_{14} &= (t_1 - t_2) (u_1 + u_2 - 3 u_3 - u_4), \\
        u_1 (u_1 + u_3)\gamma &= f_{11}, \\
        t_3 (t_1 + t_3) (\alpha + \alpha') &= f_{12}, \\
        4 (t_1^2 t_3 + u_4 u_5) (\beta + \beta') 
                  &= f_{20} + f_{15} \gamma + f_{10} \gamma^2 + f_5 \gamma^3, \\
    g_{10} c_6 c_6' &=  f_{23} + f_{18} \gamma + f_{13} \gamma^2 + f_{8} \gamma^3,
\end{align*}
where each $f_i, g_i$ is a polynomial of weighted degree $i$ in
$t_1,t_2,t_3,u_1, \ldots, u_5$, where the $t$'s have weight 2 and the
$u$'s have weight 3. These polynomials are available
from~\cite{magmafile}.  The relations were checked using
Lemma~\ref{lem:23}.  Using these relations, and the above
parametrisation of $\Sigma$, we may write the coefficients of the
quadratics $(Y - \alpha \tau)(Y - \alpha' \tau)$ and
$(Y - \beta \tau)(Y - \beta' \tau)$ as rational functions in $r$ and
$s$.  The discriminant of each quadratic is equal to $F_2(r,s,1)$
times a square, where $F_2$ is the polynomial defined in
Remark~\ref{completethesquare}.  The skew bi-invariants therefore
define a rational map from $Z(13,2)$ to the surface
$\{ y^2 = F_2(r,s,1) \}$.  Adapting the argument in
Remark~\ref{deg1-dir} shows that this map has degree $1$.  Moreover
by Theorem~\ref{myinv} we have
$j = 1728 - \beta^2/(\alpha(Q_{20} Q_{02})^{10})$, which we may then
rewrite as an element of $\Q(r,s,\sqrt{F_2(r,s,1)})$.  This gives the
moduli interpretation, and so completes the proof of
Theorem~\ref{thm:Zeqns} in the case $k=2$.

\section{Modular curves on $Z(13,1)$ and $Z(13,2)$}
\label{sec:modcrvs}

Following on from Sections~\ref{state:Z} and~\ref{sec:baran} we
describe some more modular curves on the surfaces $Z(13,1)$ and
$Z(13,2)$.

Let $m \ge 2$ be an integer coprime to $13$. Then any pair of
$m$-isogenous elliptic curves are $13$-congruent with power $k$, where
$k = 1$ if $m$ is a quadratic residue mod $13$, and $k = 2$
otherwise. There is therefore a copy of the modular curve $X_0(m)$ on
the surface $Z(13,k)$. In Table~\ref{table:mc} we explicitly identify
these curves in all cases where the genus $g$ of $X_0(m)$ is $0$ or
$1$.  The polynomials $F_1$ and $F_2$ were defined in
Remark~\ref{completethesquare}.

\begin{table}[ht]
\caption{ Copies of $X_0(m)$ on $Z(13,1)$ or $Z(13,2)$ }
\centering
\begin{tabular}{ccll} 
$m$ & $g$ & \qquad Formula specifying 
                a curve on (a blow up of) $y^2 = F_k(r,s,1)$ \\ \hline
2 & 0 & $F_2( 4, -3 - \eps^2 + t \eps^4, -2 + 2 \eps) = 2^{18} (4 t + 1) \eps^4 + O(\eps^5)$ \\ 
3 & 0 & $F_1( 1, 1 + \eps + t \eps^3, 2+ \eps ) = 16 (54 t + 1) \eps^4 + O(\eps^5)$ \\ 
4 & 0 & $F_1( 1, t \eps^2, -1 + \eps  ) = 4 (32 t + 1) \eps^2 + O(\eps^3)$ \\ 
5 & 0 & $F_2( t \eps^2, 1, -1 + \eps) = -(16 t^2 + 44 t - 1) \eps^4 + O(\eps^5)$ \\ 
6 & 0 & $F_2( \eps^3, t, \eps^2 ) = t^6 (t^2 + 34 t + 1) \eps^{18} + O(\eps^{19})$ \\ 
7 & 0 & $F_2( 1, -1 + t \eps^2  + t \eps^3, -1 + t \eps^2) = t^4 (t + 1) (t - 27) \eps^{12} + O(\eps^{13})$ \\ 
8 & 0 & $F_2( \eps^2 - \eps^3 - t \eps^4, -1, -\eps ) = (t^2 + 28 t + 68) \eps^{14} + O(\eps^{15})$ \\ 
9 & 0 & $F_1( \eps, 1 - \eps^2 +  t \eps^3, 1 ) = (t^2 - 18 t - 27) \eps^6 + O(\eps^7)$ \\ 
10 & 0 & $F_1( 0, t, 1 ) = t^4 (t - 1)^2 (t^2 + 16 t - 16)$ \\ 
11 & 1 & $F_2( 1, - \eps^3 + \eps^4 + t \eps^5   , \eps ) = (t+2) (t^3 - 14 t^2 - 12 t - 4) \eps^{20} + O(\eps^{21})$ \\ 
12 & 0 & $F_1( 1, -1 + t \eps, \eps ) = t^2 (t^2 + 14 t + 1) \eps^4 + O(\eps^5)$ \\ 
14 & 1 & $F_1( \eps, 1, -\eps^2 + t \eps^3) = (t^4 + 14 t^3 + 19 t^2 + 14 t + 1) \eps^{12} + O(\eps^{13})$ \\ 
15 & 1 & $F_2( 1, -1 + t \eps, \eps ) = (t^2 + t - 1) (t^2 + 13 t + 11) \eps^4 + O(\eps^5)$ \\ 
16 & 0 & $F_1( 1, \eps, 1 + t \eps) = (4 t^2 + 4 t - 7) \eps^2 + O(\eps^3)$ \\ 
17 & 1 & $F_1( t \eps, \eps^2, t ) = t^8 (t^4 - 6 t^3 - 27 t^2 - 28 t - 16) \eps^8 + O(\eps^9)$ \\ 
18 & 0 & $F_2( t, 0, 1 ) = t^2 + 10 t + 1$ \\ 
19 & 1 & $F_2( -1, t, 1 ) = (t - 1)^4 (t - 2) (t^3 + 10 t^2 + 12 t + 4)$ \\ 
20 & 1 & $F_2( t, 1, 1 ) = (t + 1)^6 (t^4 + 12 t^3 + 28 t^2 + 32 t + 16)$ \\ 
21 & 1 & $F_2( 1, -t^2, t ) = t^8 (t + 1)^4 (t^4 + 6 t^3 - 17 t^2 + 6 t + 1)$ \\ 
24 & 1 & $F_2( t^3 + 2 t^2, -1, t^2 + 2 t ) = t^{10} (t + 1)^6 (t + 2)^6
(t^4 - 4 t^3 - 16 t^2 - 8 t + 4)$ \\ 
25 & 0 & $F_1( t, -t, 1 ) = t^4 (t^2 + 4 t - 16)$ \\ 
27 & 1 & $F_1( t, t^2, -1 ) = t^8 (t + 2) (t^3 - 6 t^2 - 4)$ \\ 
32 & 1 & $F_2( t + 1, -t^3, t^2 + t ) = t^{20} (t + 1)^6 (4 t^4 - 12 t^2 - 16 t - 7)$ \\ 
36 & 1 &  $F_1( t + 1, -t^2 - t - 1, -t^2 - t) = 
      t^{12} (t + 1)^4 (4 t^4 + 8 t^3 + 12 t^2 + 8 t + 1)$ \\
49 & 1 & $F_1( t, -t^2, t - 1 ) = t^8 (t - 2)^2 (t - 1)^4 
                (t + 1)^2 (t^4 - 6 t^3 + 3 t^2 + 18 t - 19)$ 
\\ \hline
\end{tabular} 
\label{table:mc}
\end{table}

In compiling Table~\ref{table:mc} we used the {\tt SmallModularCurve}
database in Magma to check the moduli interpretations.  
For example, the entry with $m = 10$ shows that $Z(13,1)$ contains a
curve isomorphic to $y^2 = t^2 + 16 t - 16$. We parametrise this curve
by putting $t=-(T^2 + 8T + 20)/T$, and find by
Theorem~\ref{thm:Zeqns}(iii) that
\[ X^2 -(j + j')X + jj' = \big(X - j_{10}(T)\big)
\big(X - j_{10}(20/T) \big) \] 
where
\[j_{10}(T) = \frac{(T(T + 4)^5 + 16 T+ 80)^3}{T (T + 4)^5 (T + 5)^2}
 \]
is the $j$-map on $X_0(10)$.

To find many of these curves it was necessary to blow up the surfaces
in Theorem~\ref{thm:Zeqns}. For example, the entry with $m=15$ shows
that when we blow up our model for $Z(13,2)$ at $(r:s:1) = (1:-1:0)$
we obtain a curve isomorphic to
$y^2 = (t^2 + t - 1) (t^2 + 13 t + 11)$. Putting this elliptic curve
in Weierstrass form we find it has Cremona label $15a1$, and is
therefore isomorphic to $X_0(15)$.

The surfaces $Z(13,1)$ and $Z(13,2)$ also contain modular curves of
level $13$.  For example, the factors of $D_1$ and $D_2$ (as defined
in Theorem~\ref{thm:Zeqns}(iii)) appearing with multiplicity $13$, say
$\delta_1$ and $\delta_2$, each define a copy of the genus $2$ curve
$X_1(13)$. In fact it is a general phenomenon, exploited in \cite{KS},
that $Z(n,k)$ contains copies of $X_1(n)$ above $j= \infty$ and
$j' = \infty$.

Setting $j = j'$ in Theorem~\ref{thm:Zeqns} give a curve whose
irreducible components are modular curves of level $13$.  This gives a
convenient way of computing the double covers
$X_{\rm s}(13) \to X_{\rm s}^+(13)$ and
$X_{\rm ns}(13) \to X_{\rm ns}^+(13)$, as were recently computed by
another method in \cite{DMS}.  The details are as follows.  Recall
that in Theorem~\ref{thm:Zeqns} we wrote $Z(13,k)$ as a double cover
of $\PP^2$.  We also wrote $j + j'$ and $jj'$ as rational functions in
$r$ and $s$. We now put $r=x/z$ and $s=y/z$, and factor the numerator
and denominator of $(j-j')^2$ into irreducible polynomials in
$\Q[x,y,z]$.  Let $\Delta_k(x,y,z) = z^4 \delta_k(x/z,y/z)$.  In the
cases $k=1$ and $k=2$ we obtain
\[ (j-j')^2 = \frac{F_1 (G_1 G_2 H_1 H_2 H_3 M)^2}
{y^8 z^8 (x + y - z)^6 (x^2 + y z - z^2)^2 \Delta_1^{26}} \]
and
\[ (j-j')^2 = \frac{F_2 (G_3 G_4 G_5 H_4 H_5 H_6)^2}
{x^{10} z^{10} (x^2 + x y + x z + z^2)^4 \Delta_2^{26}} \]
where
\begin{itemize}
\item $F_1$ and $F_2$ are as in 
Remark~\ref{completethesquare}.
As noted in Section~\ref{sec:baran}, they define copies of 
$X_{\rm s}^+(13)$ and $X_{\rm ns}^+(13)$.
\item $G_1, \ldots, G_5$ have degrees $8,11,9,10,11$.
Each defines a copy of $X_{\rm s}^+(13)$ with inverse image in 
$Z(13,k)$ a copy of $X_{\rm s}(13)$.
\item $H_1, \ldots, H_6$ have degrees $8,11,13,7,10,12$. 
Each defines a copy of $X_{\rm ns}^+(13)$ with inverse image in 
 $Z(13,k)$ a copy of $X_{\rm ns}(13)$.
\item $M$ has degree $8$, but factors as the product of two quartics
  defined over $\Q(\zeta_{13})$. Each factor defines a curve whose
  inverse image in $Z(13,1)$ is a copy of $X_1(13)$, but with a
  non-standard moduli interpretation.
\end{itemize}

There is a group-theoretic explanation for the factorisations of these
numerators.  In the case $k=1$ we let $\PSL_2(\Z/13\Z)$ act on its
non-trivial elements by conjugation, and find that there are $8$
orbits, with stabilisers conjugate to (in an obvious notation)
\[ C_{\rm s}^+, C_{\rm s}, C_{\rm s}, C_{\rm ns}, C_{\rm ns}, C_{\rm ns}, B, B. \]
In the case $k=2$ we let $\PSL_2(\Z/13\Z)$ acts by conjugation 
on the elements in $\PGL_2(\Z/13\Z)$ whose determinant is not a square,
and find that there are $7$ orbits, with stabilisers conjugate to 
\[ C_{\rm ns}^+, C_{\rm s}, C_{\rm s}, C_{\rm s}, C_{\rm ns}, C_{\rm ns}, 
C_{\rm ns}.  \]

\section{Examples} 
\label{sec:ex}

\subsection{Examples over $\Q$}
We use our results, as presented in Section~\ref{sec:statres}, to
exhibit some non-trivial pairs of $13$-congruent elliptic curves
over~$\Q$.

\begin{Example}
  Let $E/\Q$ be the elliptic curve $y^2 = x^3 - 4x -3$, labelled 52a2
  in Cremona's tables.\footnote{Confusingly, the numbering of the
    elliptic curves in the isogeny class $52a$ is different in
    Cremona's tables and in the LMFDB.}  Taking $a = -4$ and $b = -3$
  we find on $X_E(13,1)$ the point
  \[ (-30:23:-72:-16:0:16:1) \]
  mapping to
  $j = -2^8 \cdot 3^3 \cdot 151^3 \cdot 2399^3 / (13 \cdot 19^{13})$.
  This is the $j$-invariant of the elliptic curve $E'/\Q$ with Cremona
  label 988b1. Therefore $E$ is directly 13-congruent to the quadratic
  twist of $E'$ by some square-free integer $d$. Comparing a few
  traces of Frobenius shows that $d=1$.
\end{Example} 

\begin{Example}
  Let $E/\Q$ be the elliptic curve $y^2 = x^3 + x - 10$, labelled 52a1
  in Cremona's tables. Taking $a = 1$ and $b = -10$ we find on
  $X_E(13,2)$ the points
\[ ( 9 : 0 : 4 : 2 : -2 : 2 : -1 ), \qquad ( 134 : -45 : 134 : 44 : 5
: -18 : 4 ), \]
mapping to $j = 2^{14} \cdot 3^3/13$ and
$-2^8 \cdot 3^3 \cdot 151^3 \cdot 2399^3/(13 \cdot 19^{13})$.  These
correspond to the elliptic curves 52a2 and 988b1 that are each skew
13-congruent to $E$.
\end{Example} 

In Theorem~\ref{thm:Zeqns} we gave equations for the surfaces
$Z(13,1)$ and $Z(13,2)$, each as a double cover of $\PP^2$.  We use
these formulae to exhibit some non-trivial pairs of $13$-congruent
elliptic curves over $\Q$, where both curves are beyond the range of
Cremona's tables.

\begin{Example}
\label{ex:Z(13,1)}
There is a $\Q$-rational point on the surface $Z(13,1)$
above $(r : s : 1) = (4 : 5 : 3)$.
This maps to the pair of $j$-invariants
\[ j = \frac{-257^3 \cdot 811^3}{2^2 \cdot 3^{12} \cdot 5^4 \cdot 11} 
  \qquad \text{ and } \qquad
j' = \frac{-441851363^3}{2^5 \cdot 3 \cdot 5 \cdot 11 \cdot 23^{13}}. \]
These are the $j$-invariants of a pair of directly $13$-congruent 
elliptic curves
\begin{align*}
 E : \qquad &  y^2 + x y + y = x^3 - 464619 x - 122105558 \\
 E' :\qquad & y^2 + x y + y = x^3 - 11276810818409 x + 14959107699948354572
\end{align*}
with conductors $N = 3778170 = 2 \cdot 3 \cdot 5 \cdot 11 \cdot 107^2$
and $N' = 86897910 = 23 N$.
\end{Example}

\begin{Example}
\label{ex:Z(13,2)}
There is a $\Q$-rational point on the surface $Z(13,2)$
above
$(r : s : 1) = ( 2 : -9 : 6)$. This maps to the pair of $j$-invariants
\[ j = \frac{29^3 \cdot 61^3 \cdot 103}{2^{19} \cdot 3^7 \cdot 17}
\qquad \text{ and } \qquad j' = \frac{13^3 \cdot 103^2 \cdot
  270539^3}{2 \cdot 3^3 \cdot 17^2 \cdot 19^{13}}. \]
These are the $j$-invariants of a pair of skew $13$-congruent elliptic
curves
\begin{align*}
 E : \qquad &  y^2 + x y = x^3 + 3796 x - 685680 \\
 E' :\qquad &  y^2 + x y = x^3 + 8246713256941 x + 11003401358367836019 
\end{align*}
with conductors $N = 1082118 = 2 \cdot 3 \cdot 17 \cdot 103^2$ and
$N' = 20560242 = 19 N$.
\end{Example}

\subsection{Examples over $\Q(t)$}
The following pairs of $13$-congruent 
elliptic curves over $\Q(t)$ 
were found as described in Section~\ref{state:Z}.
By specialising $t$ they give
rise to infinitely many non-trivial pairs of $13$-congruent 
elliptic curves over $\Q$.

\newpage

\begin{Example}
\label{QTex:dir}
Let $E/\Q(t)$ be the elliptic curve 
\[ E : \quad y^2 = x^3 - 3 p^2 q f_1 x + 2 p^2 q^2 f_2, \]
where $p(t) = t^2 - 3 t - 1$, $q(t) = t^3 - 2 t^2 - 3 t - 8$, and
\begin{align*}
f_1(t) &= t^7 + 4 t^6 + 8 t^5 + 6 t^4 - 8 t^3 - 24 t^2 - 27 t - 8, \\
f_2(t) &= t^{11} + 4 t^{10} + 5 t^9 - 12 t^8 - 63 t^7 
          - 124 t^6 - 137 t^5 - 80 t^4 - 
        61 t^3 - 72 t^2 \\ & \qquad - 153 t + 8.
\end{align*}
Taking $a = -3 p^2 q f_1$ and $b = 2 p^2 q^2 f_2$, we find on $X_E(13,1)$ 
the point
\begin{align*} (
   2592 t (t &+ 1) (t^2 + 2 t + 3) p^2 q r^2 :
   -432 t^2 (2 t^2 + 3 t + 7) p^2 q r^2 \\ & :
   -72 (t^3 + 2 t^2 + 4 t - 1) p^2 q r :
   72 p^2 q r \\ & : 
   -6 (t^6 - 3 t^4 - 11 t^3 - 14 t^2 - 5 t - 4) p :
   -24 (t - 1) p r :
   t - 1
) \end{align*}
where $r(t) = t^3 + t^2 + 2 t - 1$. 
This point maps to
$j' = p g_1^3/(rd) = 1728 + (q g_2^2)/(rd)$
where
\begin{align*}
g_1(t) &= t^{31} + 23 t^{30} + 270 t^{29} + 2379 t^{28} + 17607 t^{27}
      + 110676 t^{26} + 586710 t^{25} \\ & \qquad 
     + 2624262 t^{24} + 9977316 t^{23} + 32555542 t^{22} +
     92002244 t^{21} + 226872066 t^{20} \\ & \qquad
     + 490871649 t^{19} + 935166681 t^{18} +
     1571157252 t^{17} + 2326844467 t^{16} \\ & \qquad 
     + 3029704865 t^{15} +
     3450459162 t^{14} + 3407984048 t^{13} + 2880044002 t^{12} \\ & \qquad
     + 2037108963 t^{11} + 1159162859 t^{10} + 486247810 t^9 + 109783239 t^8
     \\ & \qquad - 25731445 t^7 - 37205624 t^6 - 17036352 t^5 - 3782272 t^4 -
     99968 t^3 \\ & \qquad + 90624 t^2 + 50176 t + 8192, \\
g_2(t) &= t^{46} + 34 t^{45} + 586 t^{44} + \ldots 
       - 10680320 t^2 - 2752512 t - 262144, \\
d(t) & = t^4 (t + 2)^3 (t^4 + 4 t^3 + 9 t^2 + 11 t + 8) (t^6 + 4 t^5 + 9 t^4 + 8 t^3 + 2 t^2 - 9 t - 6)^{13}. 
\end{align*}
This is the $j$-invariant of an elliptic curve directly $13$-congruent
to $E$.  (Notice that the full formula for $g_2$ may be recovered from
the equality of two expressions we gave for~$j'$.)  Specialising $t$
and comparing traces of Frobenius, we find that the correct quadratic
twist is
\[ E' : \quad y^2 = x^3 - 3 p^3 q^3 g_1 x + 2 p^4 q^5 g_2. \]
\end{Example}

In Table~\ref{tab:dir} we record some pairs of $13$-congruent elliptic
curves over $\Q$, obtained by specialising the parameter $t$ in
Example~\ref{QTex:dir}.  In each case we have taken simultaneous
quadratic twists so that $E$ has conductor as small as possible.  Only
the first 2 pairs of elliptic curves in Table~\ref{tab:dir} are
isogenous (via an isogeny of the degree indicated).
Example~\ref{QTex:dir} therefore shows that there are infinitely many
non-trivial pairs of directly $13$-congruent elliptic curves over
$\Q$.
\begin{table}[ht]
\caption{Pairs of directly $13$-congruent elliptic curves}
$\begin{array}{cccc|ccc}
t & E & E' & \deg & t & E & E' \\ \hline
1 & 11a3 & 11a2 & 25  & 1/4 & 7707798* & 27925352154*   \\ 
-1 & 768h1 & 768h4 & 10  & -3 & 15211515* & 1566786045*   \\ 
4 & 13688b1 & 8363368* &  & 8/5 & 46427580* & 5448415795740*   \\ 
2 & 27930s1 & 27930r1 &  & 3 & 48963840* & 42941287680*   \\ 
-4 & 80408l1 & 8282024* &  & 5 & 147656145* & 1624217595*   \\ 
-1/2 & 83030b1 & 913330* &  & 7/2 & 192105606* & 23030964786522*   \\ 
-1/3 & 271545f1 & 589524195* &  & 5/2 & 774703710* & 6034167197190*   \\ 
1/2 & 5429670* & 320350530* &  & 7 & 1040014080* & 24181367374080*  
\end{array}$
\label{tab:dir}
\end{table}
\FloatBarrier

\begin{Example}
\label{QTex:skew}
Let $E/\Q(t)$ be the elliptic curve
 \[ E : \quad y^2 = x^3 -3 p q^2 f_1 x + 2 p q^2 f_2. \]
where $p(t) = t^2 + t + 1$, $q(t) = 5 t^2 + 8 t + 11$, and
 \begin{align*}
 f_1(t) &= t^4 - 13 t^3 - 4 t^2 - 5 t + 1, \\
 f_2(t) &= 59 t^9 + 183 t^8 + 477 t^7 + 315 t^6 + 54 t^5
       - 570 t^4 - 499 t^3 - 429 t^2 - 123 t - 43. 
\end{align*}
Taking $a = -3 p q^2 f_1$ and $b = 2 p q^2 f_2$, 
we find on $X_E(13,2)$ the point 
\begin{align*} 
    ( -2 p q r_1 : p q r_2 : 2 r_3 : -24 (t + 1)^2 (t^2 + 1)^2 : t r_4 : 2 t : 0 )
\end{align*}
where
\begin{align*}
r_1(t) &= 29 t^7 + 15 t^6 + 7 t^5 - 32 t^4 + 89 t^3 + 73 t^2 + 83 t + 24, \\
r_2(t) &= 55 t^7 + 117 t^6 + 269 t^5 + 356 t^4 + 211 t^3 + 179 t^2 + 25 t + 36, \\
r_3(t) &= t^6 - 14 t^5 - 43 t^4 - 85 t^3 - 106 t^2 - 53 t - 36, \\
r_4(t) &= 13 t^5 + 34 t^4 + 17 t^3 + 11 t^2 - 22 t - 5.
\end{align*}
This point maps to $j' = p q g_1^3/d = 1728 - g_2^2/d$ where
\begin{align*}
g_1(t) &= 211 t^{14} + 665 t^{13} + 1079 t^{12} + 414 t^{11} - 1754 t^{10} -
    5658 t^9  - 9756 t^8 - 12536 t^7 
  \\ & \qquad - 12796 t^6 - 10606 t^5 - 7358 t^4
    - 4030 t^3 - 1831 t^2 - 553 t - 131, \\
g_2(t) &= 3107 t^{23} + 45563 t^{22} + 257591 t^{21} 
              + \ldots   + 35789 t + 4973,  \\
d(t) &= (t+1) (t^2+1) (2 t^2 + t + 1)^2 (2 t^3 + 2 t^2 + 3 t + 1)^{13}.
\end{align*}
This is the $j$-invariant of an elliptic curve skew $13$-congruent to $E$.  
Specialising $t$ and comparing traces of Frobenius, we find that the 
correct quadratic twist is
 \[ E' : \quad y^2 = x^3 + 3 p q^3 g_1 x + 2 p q^4 g_2. \]
\end{Example}

In Table~\ref{tab:skew} we record some pairs of $13$-congruent
elliptic curves over $\Q$, obtained by specialising the parameter $t$
in Example~\ref{QTex:skew}.  In each case we have taken simultaneous
quadratic twists so that $E$ has conductor as small as possible.  Only
the first 3 pairs of elliptic curves in Table~\ref{tab:skew} are
isogenous (via an isogeny of the degree indicated).
Example~\ref{QTex:skew} therefore shows that there are infinitely many
non-trivial pairs of skew $13$-congruent elliptic curves over $\Q$.
\begin{table}[ht]
\caption{Pairs of skew $13$-congruent elliptic curves}
$\begin{array}{cccc|ccc}
T & E_1 & E_2 & \deg & T & E_1 & E_2 \\ \hline
0 & 121c1 & 121c2 & 11  & 3 & 185900a1 & 7621900*   \\ 
1 & 162c1 & 162c4 & 21  & -7 & 255162e1 & 4848078*   \\ 
-1/3 & 1225h1 & 1225h2 & 37  & 1/2 & 1242150* & 1242150*   \\ 
-3 & 1960i1 & 21560l1 &  & 1/7 & 1695978* & 429082434*   \\ 
-2 & 14175k1 & 184275o1 &  & -3/2 & 2141594* & 49256662*   \\ 
1/3 & 23660f1 & 733460* &  & -3/5 & 2147950* & 2147950*   \\ 
-3/4 & 92950q1 & 2881450* &  & -5 & 4746924* & 507920868*   \\ 
-1/2 & 98010s1 & 98010t1 &  & 1/5 & 7495800* & 397277400*  
\end{array}$
\label{tab:skew}
\end{table}

\subsection{Tables}
In Tables~\ref{tab:pts1} and~\ref{tab:pts2} we list some $\Q$-rational
points on $Z(13,1)$ and $Z(13,2)$ that do not lie on any of the curves
of genus 0 or 1 in Section~\ref{state:Z}.  In Table~\ref{tab:all} we
list some pairs of $13$-congruent elliptic curves over $\Q$ with small
conductor.  We have 3 methods for finding such examples
\begin{itemize}
\item We sort Cremona's tables by traces of Frobenius mod $13$ and look
for matches. This method is described more fully in \cite{CF}. 
\item We loop over all elliptic curves $E$ in Cremona's tables 
(ignoring quadratic twists of earlier curves) and search
for rational points on $X_E(13,1)$ and $X_E(13,2)$.
In many cases the second elliptic curve is beyond the range of
Cremona's tables.
\item We search for rational points on $Z(13,1)$ and $Z(13,2)$.  This
  can give examples where both curves are beyond the range of
  Cremona's tables.
\end{itemize}

\begin{table}[ht]
\caption{Known rational points on $Z(13,1)$}
$\begin{array}{llll}
(2,1,-3) & (15,-63,4) & (104,-481,64) & (-5110,5329,1176)\\
(-3,4,3) & (60,-65,16) & (680,-175,289) & (6552,-1352,2835)\\
(-2,5,4) & (30,68,63) & (630,685,-324) & (-6920,8477,4800)\\
(4,5,3) & (21,-1,98) & (440,-48,1085) & (2470,9025,7436)\\
(3,1,-6) & (-51,136,111) & (495,81,-1144) & (-4389,9386,9702)\\
(-6,11,9) & (-78,172,169) & (442,1224,1183) & (7259,1525,9996)\\
(-12,16,9) & (39,169,180) & (1430,-1469,225) & (3105,13225,994)\\
(14,4,-21) & (-56,256,245) & (280,-1656,49) & (13340,4205,-5819)\\
(-5,23,6) & (-17,289,20) & (-1326,2312,1521) & (10540,289,-26908)\\
(-15,25,18) & (95,-7,418) & (3540,-3481,144) & (-34086,34385,3249)\\
(38,7,12) & (455,169,294) & (1309,-3757,588) & (8015,58166,-833)\\
(15,-40,9) & (476,289,240) & (4144,-999,2695) & (-220836,913936,859705)
\end{array}$
\label{tab:pts1}
\end{table}

\begin{table}[ht]
\caption{Known rational points on $Z(13,2)$}
$\begin{array}{llll}
(1,6,2) & (-9,40,30) & (-1176,1331,231) & (7546,1350,-735)\\
(-8,5,4) & (-40,27,24) & (1445,-216,510) & (-1682,1331,11484)\\
(1,-8,6) & (17,-56,34) & (-532,2197,1235) & (4563,12167,13455)\\
(8,-1,4) & (15,-64,20) & (63,2560,120) & (14175,-1331,4389)\\
(2,-9,6) & (-49,64,140) & (-1989,2744,2730) & (1156,-15625,5525)\\
(1,-7,9) & (175,32,140) & (2975,1,-255) & (13248,-42875,21840)\\
(11,-8,22) & (-64,343,280) & (925,-2662,4070) & (-78352,54925,21580)\\
(5,-24,14) & (153,343,-105) & (175,-4608,3080) & (25205,-98304,47712)\\
(27,-1,12) & (-363,250,165) & (1007,-4913,2584) & (-159367,109744,81016)\\
(-8,27,12) & (790,-343,1106) & (-845,4968,1482) & \\
(-20,27,30) & (-1107,824,246) & (-5635,6859,1995) & \\
\end{array}$
\label{tab:pts2}
\end{table}

Elliptic curves are specified either by their Cremona label, or by
writing $N*$ where $N$ is the conductor of the elliptic curve. In the
latter case Weierstrass equations are available from~\cite{magmafile}.
We do not list examples that could be deduced from earlier entries by
swapping over the curves, by using an isogeny of degree coprime to
$13$, or by taking simultaneous quadratic twists. We specify whether
the $13$-congruence is direct ($k=1$) or skew ($k=2$).  The entries in
bold are examples coming from the infinite families in
Examples~\ref{QTex:dir} and~\ref{QTex:skew}.

The examples where both $E$ and $E'$ are within the range of Cremona's
tables were independently found by Cremona and Freitas. Indeed, there 
are 18 such pairs 

\newpage
 
\begin{table}[ht]
\caption{Pairs of $13$-congruent elliptic curves}
$\begin{array}{ccc@{\quad}|@{\quad}ccc} 
k & E & E' & k & E & E' \\ \hline
2 & 52a1 & 988b1  & 1 & {\bf{271545f1}} & {\bf{589524195*}}  \\ 
1 & 345b1 & 10005m1  & 2 & 314330i1 & 314330j1  \\ 
2 & 735c1 & 9555h1  & 2 & 1082118* & 20560242*  \\ 
1 & 1190a1 & 265370d1  & 2 & 1137150* & 76189050*  \\ 
1 & 1274h1 & 21658t1  & 2 & {\bf{1242150*}} & {\bf{1242150*}}  \\ 
1 & 1445b1 & 10115e1  & 1 & 1296924* & 1437132*  \\ 
2 & {\bf{1960i1}} & {\bf{21560l1}}  & 2 & 1425720* & 1425720*  \\ 
1 & 3990ba1 & 43890cu1  & 2 & {\bf{1695978*}} & {\bf{429082434*}}  \\ 
2 & 4719b1 & 33033k1  & 2 & {\bf{2141594*}} & {\bf{49256662*}}  \\ 
2 & 5070j1 & 35490bg1  & 2 & {\bf{2147950*}} & {\bf{2147950*}}  \\ 
1 & 11638o1 & 151294h1  & 2 & 2164218* & 413365638*  \\ 
2 & 12274c1 & 135014s1  & 2 & 2228037* & 69069147*  \\ 
1 & {\bf{13688b1}} & {\bf{8363368*}}  & 1 & 2428110* & 31565430*  \\ 
2 & {\bf{14175k1}} & {\bf{184275o1}}  & 2 & 3647770* & 69307630*  \\ 
1 & 20184i1 & 20184j1  & 1 & 3778170* & 86897910*  \\ 
2 & {\bf{23660f1}} & {\bf{733460*}}  & 1 & 3944850* & 3944850*  \\ 
1 & {\bf{27930s1}} & {\bf{27930r1}}  & 2 & 4083510* & 730948290*  \\ 
2 & 29970f1 & 4705290*  & 2 & {\bf{4746924*}} & {\bf{507920868*}}  \\ 
2 & 69230m1 & 761530*  & 1 & {\bf{5429670*}} & {\bf{320350530*}}  \\ 
1 & {\bf{80408l1}} & {\bf{8282024*}}  & 2 & {\bf{7495800*}} & {\bf{397277400*}}  \\ 
1 & {\bf{83030b1}} & {\bf{913330*}}  & 1 & {\bf{7707798*}} & {\bf{27925352154*}}  \\ 
2 & {\bf{92950q1}} & {\bf{2881450*}}  & 2 & 10052196* & 3608738364*  \\ 
1 & 95370cl1 & 95370cm1  & 1 & {\bf{15211515*}} & {\bf{1566786045*}}  \\ 
2 & {\bf{98010s1}} & {\bf{98010t1}}  & 2 & 15893150* & 842336950*  \\ 
2 & 162266e1 & 17686994*  & 2 & {\bf{16207776*}} & {\bf{1183167648*}}  \\ 
2 & 177735a1 & 1955085*  & 2 & {\bf{17859765*}} & {\bf{553652715*}}  \\ 
2 & {\bf{185900a1}} & {\bf{7621900*}}  & 1 & 21140427* & 264234197073*  \\ 
2 & 237538k1 & 23041186*  & 1 & 21219400* & 233413400*  \\ 
2 & {\bf{255162e1}} & {\bf{4848078*}}  & 2 & 21997290* & 285964770*  
\end{array}$
\label{tab:all}
\end{table}

\FloatBarrier

\noindent
in Table~\ref{tab:all} which, together with the 3 elliptic curves (52a2,
735c2, 1190a2) isogenous to those already in the table, gives the 39
examples in \cite[Section 3.7]{CF}.  In addition the pair of curves
with conductor $1242150 = 2 \cdot 3 \cdot 5^2 \cdot 7^2 \cdot 13^2$
was independently found by Best and Matschke \cite{Best}, in
connection with their work tabulating elliptic curves with good
reduction outside $\{2,3,5,7,11,13\}$.

\appendix

\section{Elliptic curves whose $13$-torsion subgroups 
have isomorphic semi-simplifications}
           
If elliptic curves $E_1$ and $E_2$ over $\Q$ are $13$-congruent then
their traces of Frobenius are congruent mod $13$ at all primes of good
reduction. The converse is also true provided that the elliptic curves
$E_1$ and $E_2$ do not admit a rational $13$-isogeny.  Otherwise, the
Chebotarev density theorem and Brauer Nesbitt theorem only give that
$E_1[13]$ and $E_2[13]$ have isomorphic semi-simplifications.

\begin{Theorem} 
\label{thm:app}
There are infinitely many pairs of elliptic curves $E_1$ and $E_2$
over $\Q$, each admitting a rational $13$-isogeny, such that $E_1[13]$
and $E_2[13]$ have isomorphic semi-simplifications. Moreover these
examples correspond to infinitely many pairs of $j$-invariants.
\end{Theorem}

We prove the theorem by finding pairs of $13$-isogenies whose kernels
are isomorphic as Galois modules. It then follows by properties of the
Weil pairing that the $13$-torsion subgroups have isomorphic
semi-simplifications.  After we had completed the proof of
Theorem~\ref{thm:app} we discovered that essentially the same proof
was given by N. Elkies in 2013 in response to a question asked by
S. Keil at

\medskip

\noindent
\url{https://mathoverflow.net/questions/129818/elliptic-curves-over-qq}

\noindent
\url{-with-identical-13-isogeny}

\medskip

Our motivation for considering Theorem~\ref{thm:app} is the following
question of N.~Freitas, to which we still do not know an answer.

\begin{Question}[Freitas] Are there any pairs of $13$-congruent
  elliptic curves over $\Q$ where one (and hence both) of these curves
  admits a rational $13$-isogeny?
\end{Question}
 
\noindent
{\em Proof of Theorem~\ref{thm:app}.}
Let $t$ be the coordinate on $X_0(13) \isom \PP^1$ specified in 
Section~\ref{sec:state-curves}, and let $s = t+4$. 
The modular curve $X_1(13)$ is the genus $2$ curve
\[ y^2 = (x^3 - 3 x + 1)^2 - 2 (x^3 - 3 x + 1) x (x - 1) + 5 x^2 (x - 1)^2 \]
with automorphism group $G \isom C_6$ generated by $x \mapsto 1/(1-x)$ 
and $y \mapsto -y$. The forgetful map $\pi : X_1(13) \to X_0(13)$ 
quotients out by this action, and is given by 
\[ (x,y) \mapsto s = \frac{x^3 - 3 x + 1}{x (x - 1)} 
  = x + \frac{1}{1-x} + \frac{x-1}{x}. \]

  We claim that the quotient of $X_1(13) \times X_1(13)$ by the
  diagonal action of $G$ is birational to the surface
  $\Sigma = \{ Y^2 = f ( T^3 - 3T + 1,T(T-1);X)\} \subset \Aff^3$
  where
\begin{align*}
 f(\lambda,\mu;X) = (X^2 - 2 X + 5) \big((\la^2 - 2 \la \mu & + 5 \mu^2) X^2
   - 2 (\la^2 + \la \mu + 6 \mu^2) X \\ & + (5 \la^2 - 12 \la \mu + 72 \mu^2)
\big).
\end{align*}
Indeed, $G$ acts on the fibres of the map
$\alpha : X_1(13) \times X_1(13) \to \Sigma$ given by
$((x_1,y_1),(x_2,y_2)) \mapsto (T,X,Y)$ where
\[ T = \frac{ x_1 x_2 - x_1 + 1 }{x_2 - x_1}, \quad 
   X = \frac{x_2^3 - 3 x_2 + 1}{x_2 (x_2 - 1)}, \quad  
   Y = \frac{(X^2 - 3 X + 9)y_1y_2}{(x_1 - x_2)^3}.
\]
Moreover, if we define $\beta : \Sigma \to X_0(13) \times X_0(13)$ by
\[ (T,X,Y) \mapsto (s_1,s_2) = \left( \frac{(T^3 - 3 T + 1) X - 9 T (T
    - 1)} {T (T - 1) X + (T^3 - 3 T^2 + 1)},X \right). \]
then there is a commutative diagram
\[ \xymatrix{ X_1(13) \times X_1(13) \ar[rr]^{(\pi,\pi)}
  \ar[dr]_\alpha & & X_0(13) \times X_0(13) \\
  & \Sigma \ar[ur]_\beta } \]

The surface $\Sigma$ parametrises pairs of 13-isogenies together with
a choice of isomorphism between their kernels. It has infinitely many
rational points, since there is a genus $1$ fibration given by
$(T,X,Y) \mapsto T$, and it easy to exhibit a fibre (e.g. $T = -2$)
that is an elliptic curve of positive rank.  \qed

\begin{Example}
  The rational point $(T,X,Y) = (17/33, 1, 126340/33^3)$ on $\Sigma$
  corresponds to the elliptic curves
\begin{align*}
&E_1: & y^2 + y &= x^3 - x^2 - 2 x - 1, \\
&E_2: & y^2 + y &= x^3 - x^2 - 1424883795842044404862 x \\ 
&&  & \qquad \qquad \qquad \qquad \qquad ~- 20702237422068075268318817670099,
\end{align*}
with discriminants $\Delta(E_1) = -3 \cdot 7^2$ and
$\Delta(E_2) = 3 \cdot 7^2 \cdot 13^{13} \cdot 251^{13} \cdot 17681.$
By construction $E_1[13]$ and $E_2[13]$ have isomorphic
semi-simplifications.  However $E_1$ and $E_2$ are not $13$-congruent,
as may be seen from the fact that $p = 17681$ ramifies in
$\Q(E_2[13])/\Q$ but not in $\Q(E_1[13])/\Q$.
\end{Example}

\begin{Remark}
The surface $\Sigma$ also has a genus $2$ fibration given by 
$(T,X,Y) \mapsto X$. The fibres are twists of $X_1(13)$ parametrising 
$13$-isogenies with a given Galois action on the kernel.
\end{Remark}

\end{document}